\documentclass[10pt,reqno]{amsart}
\usepackage[latin1]{inputenc}
\usepackage{graphicx}
\usepackage{epsfig}
\usepackage{mathrsfs}
\usepackage{caption} 
\usepackage{amsmath}
\usepackage{amssymb}
\usepackage{color}
\newcommand{\sn}[1]{\textcolor{black}{#1}}
\newcommand{\ali}[1]{\textcolor{black}{#1}}

\def\vext{V_{\rm ext}}
\def\vint{V_{\rm int}}
\def\be{\begin{equation}}
\def\ee{\end{equation}}
\def\beq{\begin{eqnarray}}
\def\eeq{\end{eqnarray}}
\def\beqs{\begin{eqnarray*}}
\def\eeqs{\end{eqnarray*}}
\def\ea{\end{array}}
\def\ea{\end{array}}
\def\ds{\displaystyle}

\def\RR{\mathbb{R}}
\def\CC{\mathbb{C}}
\def\11{{\rm 1~\hspace{-1.5ex}1} }

\newcommand{\rfb}[1]{\mbox{\rm
   (\ref{#1})}\ifx\undefined\stillediting\else:\fbox{$#1$}\fi}

\makeatletter
\def\section{\@startsection {section}{1}{\z@}{-3.5ex plus -1ex minus
    -.2ex}{2.3ex plus .2ex}{\large\bf}}
\makeatother

\font\eufm=eufm10\font\eufms=eufm10\font\eufmss=eufm10\newfam\eufam
\textfont\eufam=\eufm\scriptfont\eufam=\eufms\scriptscriptfont\eufam=\eufmss

\usepackage{amsthm}
\swapnumbers
\newtheorem{theorem}{Theorem}[section]
\newtheorem{lemma}[theorem]{Lemma}
\newtheorem{corollary}[theorem]{Corollary}
\newtheorem{remark}[theorem]{Remark}

\newtheorem{assumption}[theorem]{Assumption}

\makeatletter
\def\section{\@startsection {section}{1}{\z@}{-3.5ex plus -1ex minus
    -.2ex}{2.3ex plus .2ex}{\large\bf}}
\makeatother

\begin{document}
\thispagestyle{empty}

\title[Schr\"odinger operator in graphs]{Dispersive effects for the Schr\"odinger equation on finite metric graphs with infinite ends}

\date\today

\author{Felix Ali Mehmeti}
\address{Universit\'e Polytechnique Hauts-de-France,
 C\'ERAMATHS/DMATHS and FR CNRS 2037, F-59313 Valenciennes, France}
\email{Felix.Ali-Mehmeti@uphf.fr}
\author{Ka\"{\i}s Ammari}
\address{LR Analysis and Control of PDEs, LR22ES03, Department of Mathematics, Faculty of Sciences of Monastir,
University of Monastir, 5019 Monastir, Tunisia}
 \email{kais.ammari@fsm.rnu.tn}
\author{Serge Nicaise}
\address{Universit\'e Polytechnique Hauts-de-France,
C\'ERAMATHS/DMATHS and FR CNRS 2037, F-59313 Valenciennes, France}
 \email{Serge.Nicaise@uphf.fr}
 
%
%
%
\begin{abstract} 
We study the free Schr\"odinger equation on finite metric graphs with infinite ends. We give sufficient conditions to obtain the
$L^1
({\mathcal R}) \rightarrow L^\infty ({\mathcal R})$ 
time decay rate at least $t^{-1/2}$.
These conditions allow certain metric graphs with circles and/or with commensurable lengths of the bounded edges.
Further we study the dynamics of the probability flow between the bounded sub-graph and the unbounded ends.  
\end{abstract}


\maketitle


\section{Introduction} \label{formulare}

A characteristic feature of the Schr\"odinger equation is the loss
of the localization of wave packets during evolution, the
dispersion. This effect can be measured by $L^{\infty}$-time decay,
which implies a spreading out of the solutions, due to the time
invariance of the $L^{2}$-norm. The well known fact that the free
Schr\"odinger group in $\RR^n$ considered as an operator family from
$L^{1}$ to $L^{\infty}$ decays exactly as $c \cdot t^{-n/2}$ follows
easily from the explicit knowledge of the kernel of this group
\cite[p. 60]{ReedSimonII}.

\medskip

In this paper we derive analogous $L^{\infty}$-time decay estimates
for the Schr\"odinger equation on some connected graphs.

\medskip

Let us now introduce some notations which will be used throughout the
rest of the paper.

$\mathcal R$ is a connected graph made of edges $e_k$, $k\in I=I_B\cup I_U$, where
$I_B$, $I_U$ are finite sets of $\mathbb{N}$ (the set of positive integers) and $I_U$ is non empty.

$\mathcal{R}_B$ is the compact core of  $\mathcal R$, i.e., the subgraph of $\mathcal R$  made of the edges $e_k$, $k\in I_B$.

$\mathcal{R}_U=\mathcal{R}\setminus \mathcal{R}_B$  corresponds to the non-compact part of $\mathcal R$ that is made of a finite number of infinite star graphs, edges of $\mathcal{R}_U$ are called outer edges.

For all $k\in I_B$, $e_k$ has a finite length $\ell_k$, hence $e_k$ is identified with $(0,\ell_k).$

For all $k\in I_U$, $e_k$ has a infinite length, hence $e_k$ is identified with $(0,+\infty).$ 

$V$  is the set of vertices of $\mathcal R$.

For all $v\in V$, $I_v=$ set of edges having $v$ as vertex.

$V_{\rm ext}$  is the set of exterior vertices of $\mathcal R$ ($\# I_v=1$), while
$V_{\rm int}$  is the set of interior vertices of $\mathcal R$ ($\# I_v\geq 2$).

Finally let $V_{\rm connecting}$ be the set of connecting vertices of $\mathcal R$, namely the \ali{set} of $v\in V_{\rm int}$ such that $I_v\cap I_U\ne\emptyset$. Without loss of generality, we can assume that for any vertex $v \in V_{\rm connecting}$ \ali{we have} $\# I_v\geq 3$.

By convention, the vertices of $V_{\rm int}$  
remain interior vertices of   $\mathcal{R}_B$, hence the set of exterior vertices of $\mathcal{R}_B$
is $V_{\rm ext}$.

\begin{center} \label{fig1}
\includegraphics[scale=0.60]{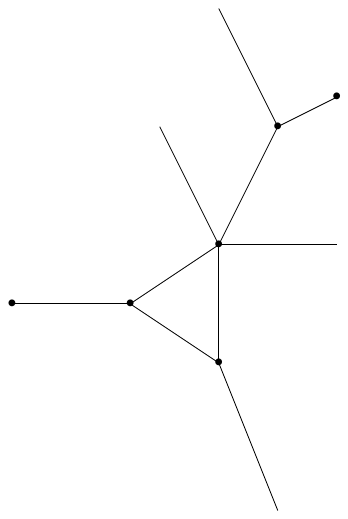}
 \put(-2,119){$v_1$}
 \put(-112,59){$v_2$}
 \put(-72,65){$v_3$}
 \put(-40,45){$v_4$}
 \put(-40,73){$v_5$}
  \put(-24,108){$v_6$}
\captionof{figure}{A graph with $\# I_B = 6$ and $\# I_U = 4$\label{exemplegraph}}
\end{center}

An example of a such a graph is presented in Figure \ref{exemplegraph}, \ali{the} vertices are represented by a dot, and   $V_{\rm ext}=\{v_1,v_2\}$,  
$V_{\rm int}=\{v_3,v_4,v_5,v_6\}$, $V_{\rm connecting}=\{v_4,v_5,v_6\}$. Consequently, $\mathcal{R}_B$ is made of   edges joining $v_1v_6$, $v_6v_5$, $v_5v_4$, $v_5v_3$, $v_3v_4$
and $v_3v_2$ (hence $\# I_B = 6$), and  \ali{$\mathcal{R}_U$} is made of four half-lines (one \ali{emanating} from $v_4$, two emanetating from $v_5$ and one \ali{emanating} from $v_6$).

Introduce ${\mathcal H}= L^2 (\mathcal{R}) =   \ds \ali{\prod_{k\in I}} L^2(e_k)$ the Hilbert space \ali{product} of the $L^2(e_k), k\in I$
with inner product
$$((u_k),(v_k))_{\mathcal H} =   \sum_{k\in I}
(u_k,v_k)_{L^2(e_k)}.$$

Let $H : {\mathcal D}(H) \subset {\mathcal H} \rightarrow {\mathcal H}$ be the linear operator on
${\mathcal H}$ defined by :
\beqs
{\mathcal D}(H) = \Big\{u=(u_k) \in C({\mathcal R})\cap \ds \ali{\prod_{k\in I}} H^2(e_k)
\hbox{ satisfying } 
\\
u(v)=0,  \forall  v \in \vext, \,
\sum_{k\in I_v} \frac{d u_k}{dn_k} (v)=0,  \forall  v \in \vint\Big\},
\eeqs
where, as usual, $\frac{d u_k}{dn_k} (v)$ means the unit outward normal derivative of $u_k$ at $v$, and
$$
H \ali{u} = \left(- \frac{d^2 u_k}{dx^2} \right)_{k\in I} = - \Delta_{{\mathcal R}}\ali{u}, \forall u\in {\mathcal D}(H).
$$
This operator $H$ is self-adjoint and
its spectrum $\sigma(H)$ is equal  to $[0,+\infty)$ due to Theorem 5.2 of \cite{Solomyak:03}.
The self-adjointness and non-negativity of $H$ can be shown by Friedrichs extension.

\sn{\begin{remark} {\rm
The results stated below remain valid if we replace the  Dirichlet boundary conditions at the exterior vertices
\[
u(v)=0,  \forall  v \in \vext,
\]
 by the
Neumann 
boundary conditions at the exterior vertices
\[
\frac{d u_k}{dn_k} (v)=0,  \forall k\in I_v, v \in \vext.
\]
Similarly mixed boundary conditions at the exterior vertices, i. e., Dirichlet ones at a subset of $\vext$ and Neumann ones at the remaining set, can 
be considered.
}
\end{remark} 
}

Let us also introduce the Laplace operator on  $\mathcal{R}_B$ with Dirichlet boundary conditions as follows:
\beqs
{\mathcal D}(\Delta_{\mathcal{R}_B}^{\rm Dir}) = \Big\{u=(u_k) \in C({\mathcal R}_B)\cap \ds \ali{\prod_{k\in I_B}} H^2(e_k)
\hbox{ satisfying } 
\\
u(v)=0,  \forall  v \in \vext,\\
\sum_{k\in I_v} \frac{du_k}{dn_k} (v)=0,  \forall  v \in \vint \Big\},
\eeqs
$$
-\Delta_{\mathcal{R}_B}^{\rm Dir} \ali{u} = \left(- \frac{d^2 u_k}{dx^2} \right)_{k\in I_B}.
$$
Again $-\Delta_{\mathcal{R}_B}^{\rm Dir}$ is a non-negative self-adjoint operator, but here  with a compact resolvent, hence its spectrum $\sigma(-\Delta_{\mathcal{R}_B}^{\rm Dir})$ is a discrete set.

Here, under some technical \ali{but crucial} assumptions  (described below), we prove that the free Schr\"odinger group $(e^{-itH})_{t\in \mathbb{R}}$ on the graph 
${\mathcal R}$ satisfies the standard $L^1-L^\infty$ dispersive
estimate. 
This means that, under these assumptions, the time decay rate is the same as in the case of a line  \cite{ReedSimonII,weder,GS},  a half-line  
\cite{wederb}, a star
shaped network  \cite{admi1,banica,banicab,ignat,amamnic} and a tadpole graph \cite{amamnicbis}. 

Note further that  Ammari \& Sabri have proved 
in \cite{amsab1}  dispersive estimates for the adjacency matrix on a discrete regular tree and the Schr\"odinger equation on a metric regular tree with uniform potential on its edges/vertices. The latter model is an extension of periodic Schr\"odinger operators on the real line. They have established a sharp decay rate in $t^{-3/2}$ for both models, providing the first-order asymptotics. Additionally, in another work \cite{amsab2}, they have demonstrated a sharp dispersive estimate \[\left\|e^{i t H}f\right\|_\infty \leq t^{-d/3} \left\|f\right\|_1,\]
for any Cartesian product of the integer lattice and a finite graph, denoted as $\mathbb{Z}^d \square G_F$. This result applies to various structures like infinite ladders, k-strips, and infinite cylinders, including those with specific potentials.
  
\medskip

The present paper is organized as follows. The kernel of the resolvent
needed for the proof of Theorems \ref{mainresult} \ali{- \ref{mainresultfrequency bandYB}}, is given in
section \ref{sec2}. 
The calculation of the resolvent is equivalent to the resolution of a linear system which describes the impact of the transmission conditions on the coefficients of the scattering matrix.
In the general setting of this section we assume that determinant of the matrix of this system factorizes in a factor whose zeros are the eigenvalues of the Hamilton operator and a nonzero factor together with hypotheses on the structure of the Cramer solution formula. These hypotheses 
\ali{(Assumption \ref{lassumptionserge})}
imply the existence of a
decomposition of the kernel of the
resolvent into a \sn{term with
poles on the positive real axis (\ali{which} in some particular cases is meromorphic in the whole complex plane with
poles on the positive real axis)}
and a term which is continuous on
the real line but \ali{possibly} discontinuous when crossing it.

%
\ali{
In section \ref{sec2}  we further introduce the orthogonal decomposition of $L^2(\mathcal{R})$
into the span $Y_D$ of the eigenfunctions of $H$, the span $Y_k$
of the functions with support in $e_k, k \in I_U$ and the orthogonal complement $Y_B$ of the sum of the preceeding spaces. This makes sense because all eigenfunctions of $H$ have their support in $\mathcal{R}_B$ and are thus orthogonal to all $Y_k$.
In this context we suppose that the above mentioned Cramer solution formula for the resolvent of $H$ has only isolated poles of degree 1 (assumption \ref{lassumptionsergeYj}).
We may suppose also the absence of the term in the Cramer solution formula having these poles (assumption \ref{lassumptionsergeYB}). Thanks to this spectral- and support-based hybrid approach, we obtain different expressions for the kernel of the resolvent of $H$, valid when applied to elements of $Y_k$ (under the assumptions \ref{lassumptionserge} and \ref{lassumptionsergeYj}) and when applied to elements of $Y_B$ 
(under the assumptions 
\ref{lassumptionserge} ans \ref{lassumptionsergeYB}). These are the theorems \ref{res.idYj} and \ref{res.idYB}. As a corollary we obtain an expression for the resolvent valid when applied to all elements of $H$ under the hypotheses 
\ref{lassumptionserge},  
\ref{lassumptionsergeYj}
and 
\ref{lassumptionsergeYB}. 
Finally we see in theorem  
\ref{absence_sing_spec}
that under these last hypotheses we have the absence of a singular continuous spectrum. 
}%
\ali{
For an interpretation of this last phenomenon we look at the existing literature on the singular continuous spectrum, eg. \cite{pearson}. The author constructs Schr\"odinger equations with an infinity of potential bumps, which are distributed irregularly on the real line. This creates a spectrum which tends to be concentrated along Cantor-type sets when adding successively the bumps, leading to a singular continuous spectral measure. In our case we have only a finite number of finite edges, which seem to be unable to create a similar effect.
}%
\ali{
Finally we note that the pure point spectrum (created by the eigenfunctions supported on the finite subgraph) is immersed into the continuous spectrum (created by the scattering states which dissapear into the infinite edges) because, 
by arguments using the classical Lax-Milgram theory, the total spectrum of $H$ is equal to $[0,\infty[$. 
}%
%
Note that in \cite{cacc,nojaetall:15} the eigenvalues and eigenvectors are given without this detailed analysis of the rest of the spectrum; there the authors further consider the bifurcation of solutions of the nonlinear  Schr\"{o}dinger equation from the eigenvectors of the linear operator.

\medskip

In section \ref{proofm} we give the proofs of the main results of the
paper (Theorems \ref{mainresult} \ali{- \ref{mainresultfrequency bandYB})}. To this end we replace the formula
of the resolvent into Stone's formula and prepare all terms for the
application of the Lemma of van der Corput. For this section we assume that the second factor of the matrix mentioned above is bounded away from zero. This hypthesis has an impact on the geometry of the metric graph as can be seen in section 5. It implies that the inverse of the second factor is a quasiperiodic function and can thus be developed in a trigonometric series thanks to a theorem of harmonical analysis.
This leads to a decomposition of the solution formula in terms which can be treated by the van der Corput Lemma without amplitude, implying a time decay-rate at least as $t^{-1/2}$.

\medskip

Section 4 is devoted to the analysis of the evolution of the local probability density in terms of the probability flow. The Schrödinger group is a group of isometries. This is linked to the interprétation of $|u|^2$ as the probability density of the position of the particle. Due to this isometry property, probability behaves as an incompressible fluid, which leads to the possibility to define a probability flow. Using this idea, we prove that the rate of change
of the probability of the position of the particle to be in the sub-graph $\mathcal{R}_B$ made of the bounded edges  is equal to $i$ times the flow from $\mathcal{R}_B$ to its complement $\mathcal{R}_U$. This type of equation is in general called continuity equation
(for the basic concept see \cite{Likharev}, section 1.4).
     It allows us to prove the boundedness of the  rate of change of the probability in $\mathcal{R}_B$. For the proof we use the fact that the values of the flow at the contact vertices between $\mathcal{R}_B$ and $\mathcal{R}_U$ can be controlled by the $H^2-$norm of the solution on $\mathcal{R}_U$, using the Sobolev embedding $H^2(e_k) \subset C^1(\overline{e_k})$. 
\\ 
Using similar arguments we prove that the rate of change of the probability in $\mathcal{R}_B$ decays at least as $t^{-1}$

\medskip

 In the last section we consider several special metric graphs with semi-infinite ends. These examples show that our hypotheses to obtain the 
time decay-rate at least as $t^{-1/2}$
include the possibility of circles (made of one or several edges). 
In the study of the Y-graph and the tadpole with two heads we show that our hypotheses for time-decay require that the two bounded edges have commensurable lengths.
\ali{
Especially we use that in the non-commensurable case there exists a sequence of rational numbers $\frac{p}{q}$ approaching the ratio of the lengths of the two bounded edges at least as rapidly as const$\cdot \frac{1}{q}$ tends to zero (\cite{Hartman-49}). This implies that for a sequence of frequencies $z$ related to the sequence of the $q$'s, the kernel $K(x,y,z)$ of the resolvent 
tends to infinity while it does not have any singularities for finite $z'$s. As a related consequence (\ref{lowerbddc}) is not satisfied and we are thus not able to prove the decay at least as 
c$\cdot t^{-\frac{1}{2}}$ for the free Schrödinger group
($L^{1}$ to $L^{\infty}$).
For an interpretation one might think that the above mentioned asymptotic behavior of the kernel of the resolvent for large frequencies prevents the existence of a factor $c$ which is valid for all frequency bands when using stationary phase-type arguments. Intuitively the irrationality of the ratio of lengths of the edges makes it more difficult that signals turning round in the finite graph cancel out during evolution.
}
For the tadpole with two heads we find nice \ali{expressions} for the resolvent \ali{and} for the solution on the unbounded edge, which are valid regardless the ratio of the lengths of the two bounded edges.  
Note that the solution formula includes all frequencies up to infinity, a result which is inaccessible in the general context of sections 2 and 3.

\section{Kernel of the resolvent} \label{sec2}

Given $z\in \CC_+:=\{z\in \CC: \Im z>0\}$ and $f\in L^2({\mathcal R})$, we are looking 
for $u\in \mathcal{D}(H)$ solution of
\be\label{systemresolvante}
\Delta_{{\mathcal R}} u+z^2 u=f \hbox{ in } {\mathcal R}.
\ee

Hence as in \cite[§ 3.2]{banica} we look for $u$ in the form
\beq\label{sergeu1}
u_k(x)&=& c_k\sn{(z,f)} e^{-izx}+d_k\sn{(z,f)}  e^{izx}
\\
&+&  \frac{1}{2iz}\int_0^{\ell_k} e^{iz |x-y|} f_k(y) \,dy, \forall x\in (0, \ell_k), k\in I_B,
\nonumber
\\
\label{sergeu2}
u_k(x)&=& d_k\sn{(z,f)}  e^{izx}\\
&+&  \frac{1}{2iz}\int_0^{\infty} e^{iz |x-y|} f_k(y) \,dy, \forall x\in (0, +\infty), k\in I_U,
\nonumber
\eeq
where $f_k$ is the restriction of $f$ to the edge $e_k$; $c_k\sn{(z,f)}$, $d_k\sn{(z,f)}$  are constants (that depend on $z$ and $f$) fixed
in order to satisfy the transmission conditions. 
Indeed from these expansions, we clearly see that
for any  $k\in I$, one has:
\[
u_k''+z^2 u_k= f_k \hbox{ in } e_k.
\]
Now we see that the  transmission conditions are equivalent to
\be \label{systserge}
\mathcal{D}(z) ((c_k,d_k)_{k\in I_B}, (d_k)_{k\in I_U})^\top=b(z, f),
\ee
where $\mathcal{D}(z)$ is a square matrix of size $N=2\# I_B+\# I_U$
whose entries are linear combinations of 1, and $e^{\pm iz \ell_k}$ (with $k\in I_B$) 
with coefficients independent of $z$, \sn{more precisely 
the entries of a given column of the matrix $\mathcal{D}(z)$ are
either 0, or $\pm 1$, or $\pm e^{\pm iz\ell_k}$ or 
$\pm 1 \pm e^{\pm iz\ell_k}$, for some $k \in I_B$. Finally,
the right-hand side $b(z, f)=R F(z,f)$, where $R$ is $N\times N$ matrix whose entries are either $0$, or $\pm 1$ and $F(z,f):=\left((F_{k+}(z,f), F_{k-}(z,f))_{k\in I_B}, (F_{k+}(z,f))_{k\in I_U}\right)^\top$  
 is a column of size $N$ whose entries are given 
 \[
F_{k+}(z,f)= \frac{1}{2iz}\int_0^{\ell_k} e^{iz y} f_k(y) \,dy, \quad F_{k-}(z,f)=
\frac{1}{2iz}\int_0^{\ell_k} e^{iz (\ell_k-y)} f_k(y) \,dy, \forall k\in I_B,
 \]
 and 
 \[
 F_{k+}(z,f)=\frac{1}{2iz}\int_0^{\infty} e^{iz y} f_k(y) \,dy, \forall k\in I_U.
 \]}
Hence a solution to \rfb{systserge} exists if and only if
 the determinant of $\mathcal{D}(z)$ is different from zero. According to Theorem 3.2 of \cite{ExnerLipovsky-07}, the zeroes of $\mathcal{D}(z)$ correspond to the scattering resonances of 
 the graph (see \cite{BK,Colin-Truc:16}), which are situated in $\Im z\leq 0$. Nevertheless, let us analyze this determinant,  for any $z\in \CC_+\cup \RR^*$, with  $\RR^*=\RR\setminus\{0\}$.  For that purpose, we introduce a subset $S_d$ of $\sigma(-\Delta_{\mathcal{R}_B}^{\rm Dir})$ defined as the set of $\lambda\in \sigma(-\Delta_{\mathcal{R}_B}^{\rm Dir})$ such that there exists an eigenvector $\varphi\in D(\Delta_{\mathcal{R}_B}^{\rm Dir})$ \ali{associated with $\lambda$} such that
 \begin{equation}\label{eq:serge:10/09:1}
 \varphi(v)=0, \forall v\in V_{\rm connecting}.
 \end{equation}
 As $V_{\rm connecting}$ is not empty, we deduce  that 
 $0$ cannot be in $S_d$.

  For future uses, if $S_d$ is non empty, we denote by $S_d:=\{\lambda_m^2\}_{m\in  \mathbb{N}}$, with $\lambda_m>0$, where $\lambda_m^2$ is repeated according to its multiplicity (corresponding to the number of associated independent eigenvectors $\varphi$ satisfying \eqref{eq:serge:10/09:1}). For all $m$, we then denote by 
  $\varphi^{(m)}$ the   eigenvector of $-\Delta_{\mathcal{R}_B}^{\rm Dir}$ associated with $\lambda_m^2$ (and satisfying \eqref{eq:serge:10/09:1}).
 \sn{Note that the extension of $\varphi^{(m)}$ by zero on $\mathcal{R}_U$, still denoted by  $\varphi^{(m)}$ for simplicity, belongs to $D(H)$
 and is an eigenvector of $H$ of eigenvalue $\lambda_m^2$. We will see below that this yields all the eigenvectors of $H$.}
  
 \begin{lemma}
 \label{lspectrediscret}
 It holds
 $\det \mathcal{D}(z)=0$ with $z\in \CC_+\cup \RR^*$ if and only if
 $z^2$ belongs to   $S_d$.
 \end{lemma}
 \begin{proof}
\ali{We have},  $\det \mathcal{D}(z)=0$ if and only if the homogeneous system associated with
 \eqref{systserge} has a non-trivial solution, hence if and only if 
 a non-trivial solution $u$ of \eqref{systemresolvante} in the form \eqref{sergeu1}-\eqref{sergeu2} with $f=0$ exists.
 In this case, as $u$ satisfies
\[
u_k''+z^2 u_k= 0 \hbox{ in } e_k, \forall k\in I,
\]
by multiplying this identity by $\bar u_k$,  integration by troncating the infinite edges up to a positive real number $R$, we have
 \[
\sum_{k\in I_B} \int_{e_k} (u_k''+z^2 u_k)\bar u_k\,dx+\sum_{k\in I_U} \int_{0}^R (u_k''+z^2 u_k)\bar u_k\,dx=0.
\]
Hence by integration by parts and using the transmission conditions, we find that
\[
\sum_{k\in I_B} \int_{e_k} (|u'_k|^2-z^2 |u_k|^2)\,dx+\sum_{k\in I_U} \int_{0}^R (|u_k'|^2-z^2 |u_k|^2)\,dx
-\sum_{k\in I_U} u_k'(R)  \bar u_k(R) =0.
\]
But due to  \eqref{sergeu2} with $f=0$ we have
\[
u'_k=iz u_k, \forall k\in I_U,
\]
and therefore the previous identity becomes
\[
\sum_{k\in I_B} \int_{e_k} (|u'_k|^2-z^2 |u_k|^2)\,dx+\sum_{k\in I_U} \int_{0}^R (|z|^2-z^2) |u_k|^2\,dx
-iz\sum_{k\in I_U} |u_k(R)|^2 =0.
\]
Writting $z=a+ib$ with $b\geq 0$, we get equivalently
 \beq\label{serge24/05:1}
 \sum_{k\in I_B} \int_{e_k} (|u'_k|^2+(b^2-a^2-2iab) |u_k|^2)\,dx
 \\
 +2(b^2-iab)\sum_{k\in I_U} |d_k|^2 \int_{0}^R   e^{-2bx}\,dx
-i(a+ib)\sum_{k\in I_U} |d_k|^2  e^{-2bR} =0.\nonumber
 \eeq
 If $a=0$, then taking the real part of this identity, we get
 \beqs
 \sum_{k\in I_B} \int_{e_k} (|u'_k|^2+b^2 |u_k|^2)\,dx
 \\
 +2b^2\sum_{k\in I_U} |d_k|^2 \int_{0}^R   e^{-2bx}\,dx
+b\sum_{k\in I_U} |d_k|^2  e^{-2bR} =0.\nonumber
 \eeqs
 And since $b>0$, we deduce that $u$ is zero, which is impossible.
 
If $a\ne 0$,  by taking the imaginary part of the identity \eqref{serge24/05:1}, we find 
 \be\label{serge24/05:3}
 2b  \sum_{k\in I_B} \int_{e_k}  |u_k|^2\,dx
 +2b\sum_{k\in I_U} |d_k|^2 \int_{0}^R   e^{-2by}\,dx
+e^{-2bR}  \sum_{k\in I_U} |d_k|^2  =0.
\ee
We now distinguish the case $b>0$ from the case $b=0$.
 \\
 1. If $b>0$, then by direct calculation of the second term of   the previous identity we get equivalently
 \[
 2b  \sum_{k\in I_B} \int_{e_k}  |u_k|^2\,dx
 +\sum_{k\in I_U} |d_k|^2=0.
 \]
 Consequently $d_k=0$ for all $k\in I_U$
 and 
 \[
 u_k=0 \hbox{ on } e_k, \forall k\in I_B,
 \]
 which directly implies that
$c_k=d_k=0$ for all $k\in I_B$. Hence $\det \mathcal{D}(z)$  cannot be equal to zero.\\
 2. If $b=0$, the  identity \eqref{serge24/05:3} directly implies that
 \[
 \sum_{k\in I_U} |d_k|^2 =0,
 \]
 or equivalently
 \be
\label{serge24/05:2}
 d_k=0, \forall k\in I_U.
 \ee
 This means that $u$ is identically equal to zero on the outer edges of the graph,
 and hence $z^2=a^2$ is an eigenvalue of the Laplace operator on the graph $\mathcal{R}_B$ 
 with Dirichlet boundary conditions at all exterior vertices $v\in V_{\rm ext}$ with eigenvector $u$.
In other words, $z^2$ belongs to the set of eigenvalues $\sigma(-\Delta_{\mathcal{R}_B}^{\rm Dir})$.
 Since additionally, $u=0$ on the outer edges of the graph,
$u$ is zero at  the vertices $v\in I_{\rm connecting}$, and therefore $z^2$ belongs to $S_d$. 
 \end{proof}

 Lemma \ref{lspectrediscret} suggests that the point spectrum corresponds to $S_d$,
 this is indeed the case.


 \begin{corollary} \label{lemma1}
 With $S_d$ introduced before, it holds
 \[
 \sigma_{pp}(H)=S_d.
 \]
\end{corollary}
\begin{proof}
From the previous Lemma, any element $z^2$ of $S_d$ belongs to $\sigma_{pp}(H)$.
On the contrary let  $\lambda> 0$ be an eigenvalue of $H$ of eigenvector
$u$, then
for any $k\in I$,
$u_k\in L^2(e_k)$ is \ali{of} the form
\[
u_k(x)=c_k e^{-izx}+d_k  e^{izx}, \forall x\in e_k,
\]
with  $z=\sqrt{\lambda}$ and $c_k, d_k$ are complex number fixed to satisfy
the transmission conditions. From the $L^2$ property of $u_k, k\in I_U$, we deduce that
\be\label{serge24/05:4}
c_k=d_k=0, \forall k\in I_U.
\ee
Indeed for any $R>0$, direct calculations yield
\[
\int_0^R |c_k e^{-izx}+d_k  e^{izx}|^2\,dx=  R(|c_k|^2+|d_k|^2)+2\Re \left(c_k\bar d_k \frac{1}{2iz}  (e^{2izR}-1) \right).
\]
Hence by Cauchy-Schwarz's and Young's  inequalities, we get
\[
\int_0^R |c_k e^{-izx}+d_k  e^{izx}|^2\,dx\geq  \left(R-\frac{1}{z} \right) \left(|c_k|^2+|d_k|^2 \right).
\]
Letting $R$ goes to infinity, we find that \eqref{serge24/05:4} holds, since the left-hand side of this estimate is uniformly bounded in $R$.
This implies that $u$ is identically equal to zero on the infinite edges and the arguments of the proof of  Lemma \ref{lspectrediscret}
imply that $\lambda$ belongs to $S_d$.

In the case $\lambda=0$, $u$ satisfying
\[
u_k''=0 \hbox{ on } e_k, \forall k\in I_U,
\]
we get
 \[
\sum_{k\in I} \int_{e_k} u_k'' \bar u_k\,dx=0,
\]
and by integration by parts and using the transmission conditions, we arrive at
\[
\sum_{k\in I} \int_{e_k} |u_k'|^2\,dx=0,
\]
hence $u$ is constant on each edge. By the $L^2$ property of $u$, we then deduce that
\[
u_k=0 \hbox{ on } e_k, \forall k\in I_U,
\]
and by the continuity property, we conclude that $u$ is zero
and therefore $0$ is not an eigenvalue of $H$.
\end{proof}

This result is in accordance with Corollary 2.1 of \cite{Colin-Truc:16}.

The previous results  also allow to deduce the kernel of the resolvent.

\begin{theorem} \label{theo1}
Let   $f \in \mathcal H$. Then, for $x \in {\mathcal R}$ and $z\in \CC$ such that
$\Im z> 0$, recalling that $R(z^2, H)=(z^2\mathbb{I}-H)^{-1}$, we have
\be\label{resolvantformula} [R(z^2, H)f](x)= \int_{{\mathcal R}} K(x, x', z^2) f(x') \; dx',
\ee
where, by setting $d(z)=\det \mathcal{D}(z)$,
the kernel $K$ is defined as follows:
\begin{equation}
\label{kernel11}  K(x,y,z^2) =
    \frac{1}{2iz}  \left(e^{iz|x-y|}+ \frac{1}{d(z)} \sum_{\sigma=\pm,\, \tau=\pm}   c_{k,\sigma, \tau}(z) e^{iz (\sigma x+\tau y)}
\right),   \forall x, y \in  e_k, k\in I_B, 
\end{equation}    
\begin{equation}
\label{kernel11b}  \sn{K(x,y,z^2) =
    \frac{1}{2izd(z)}   \sum_{\sigma=\pm,\, \tau=\pm}   c_{k,k',\sigma, \tau}(z) e^{iz (\sigma x+\tau y)},   \forall x \in  e_k, y\in e_{k'}, k, k'\in I_B, k\ne k',}
\end{equation}    
\begin{equation}
K(x,y,z^2) = \frac{1}{2iz d(z)} \sum_{\sigma=\pm}   c_{k,k',\sigma}(z) e^{iz (\sigma x+y)} ,  \forall x\in e_k, y\in  e_{k'},  k\in I_B,  k' \in I_U,
\label{kernel12}      
\end{equation}
\begin{equation}
 K(x,y,z^2) = \frac{1}{2iz} \left(e^{iz|x-y|}+ \frac{c_{k}(z)}{d(z)}     e^{iz (x+y)}\right),  \forall x, y \in  e_k, k\in I_U,
\label{kernel22}
 \end{equation}
 \sn{\begin{equation}
K(x,y,z^2) = \frac{1}{2iz d(z)} \sum_{\tau=\pm}  c_{k,k'}(z) e^{iz (x+y)},  \forall x\in  e_k, y\in e_{k'},  k, {k'} \in I_U, k\ne k',
\label{kernel22b}      
\end{equation}}
\begin{equation}
K(x,y,z^2) = \frac{1}{2iz d(z)} \sum_{\tau=\pm}  c_{k,k',\tau}(z) e^{iz (x+\tau y)} ,  \forall x\in  e_k, y\in e_{k'},  k \in I_U,   k'\in I_B,
\label{kernel21}      
\end{equation}
where the functions $c_{k,\sigma, \tau}, \sn{c_{k,k',\sigma, \tau}}, c_{k,k',\sigma},  c_{k}, \sn{c_{k,k'}}$ and $c_{k,k',\tau}$ are \sn{holomorphic functions that are
 polynomials on the variables  $e^{\pm iz \ell_j}$, $j\in I_B$} with coefficients independent of $z$ but that may depend on $k, k', \sigma$ and $\tau$.
\end{theorem}
 \begin{proof}
 By the previous Lemma, $d(z)$ is different from zero for any $z\in \CC$ such that
$\Im z> 0$, therefore \sn{the linear system
(\ref{systserge}) has a unique solution $X(z,f)$ given by
\[
X(z,f)=\mathcal{D}(z)^{-1} R F(z,f).
\]
Hence for a vector $X=((c_k,d_k)_{k\in I_B}, (d_k)_{k\in I_U})^\top\in \mathbb{C}^N$, if we introduce
the linear mappings
\beqs
P_jX= \ali{(c_j, d_j)}^\top, \forall j\in I_B,
\\
P_jX=d_j, \forall j\in I_U,
\eeqs
we deduce that the solution $u$ of 
\eqref{systemresolvante}  in the form
\eqref{sergeu1}-\eqref{sergeu2} is given by
\beqs
u_k(x)&=& \ali{(e^{-izx} , e^{izx})} P_k\mathcal{D}(z)^{-1} R F(z,f)
+  \frac{1}{2iz}\int_0^{\ell_k} e^{iz |x-y|} f_k(y) \,dy, \forall x\in (0, \ell_k), \forall k\in I_B,
\\
 u_k(x)&=&   e^{izx} P_k\mathcal{D}(z)^{-1} R F(z,f)
+  \frac{1}{2iz}\int_0^{\infty} e^{iz |x-y|} f_k(y) \,dy, \forall x\in (0, +\infty), \forall k\in I_U.
\nonumber
\eeqs
Using that $\mathcal{D}(z)^{-1}=d(z)^{-1}\operatorname{adj}(\mathcal{D}(z))$,  the form of the coefficients of the matrices $\mathcal{D}(z)$ and $R$
and of the right-hand side $F(z,f)$, we arrive at \eqref{resolvantformula}
with kernels in the announced forms.}
  \end{proof}

As usual, to obtain the resolution of the identity of $H$, we want to use the limiting absorption principle
that consists \ali{in passing} to the limit in $K(x,y,z^2)$ as $\Im z$ goes to zero.
But in view of the presence of the factor $d(z)$  in the denominator of the kernels,
this limit is a priori not allowed \sn{if $S_d$ is not empty as $d$ will} have zeroes on the real line. 
\sn{In such a case} \ali{we try to} split up the kernel $K(x,y,z^2)$ into a \sn{part $K_p$ with
poles at the points $\pm\lambda_{m}$ but that, in practice, is continuous on some  closed subspaces of $L^2(\mathcal{R})$} and a part $K_c$ which is
continuous on the real line but \ali{possibly}  discontinuous when crossing it.
To be able to do so, we make the following assumption, \ali{which is} satisfied for the operator $H$ on a large class of graphs $\mathcal{R}$.

\begin{assumption}\label{lassumptionserge}
\ali{Let us recall that the $\ell_j$ are the lengths of the finite edges $e_j$ of our metric graph}.
The determinant $d(z)$ can be factorized in the form
\be\label{factdetD(z)}
d(z)=d_d(z)d_c(z), \forall z\in \CC_+\cup \RR,
\ee
with two   holomorphic functions $d_d, d_c$
 on $\CC$ that are  
\sn{polynomials \ali{in} the variables  
\ali{$e^{iz r_{j}^+\ell_j}$, $e^{-iz r_{j}^-\ell_j}$}, $j\in I_B$ and 
for some positive $\ali{r_{j}^+, r_{j}^-} \in \mathbb{Q}$,} with coefficients independent of $z$
 and
 such that
 \[
 d_d(z)=0 \Leftrightarrow z^2\in S_d,
 \]
and therefore
\[
d_c(z)\ne 0, \forall z\in \RR.
\]
There exist two matrix valued
functions
\beqs
\mathcal{M}_c: \CC_+\cup \RR^* \to \CC^{N\times N}: z \mapsto \mathcal{M}_c(z),
\\
\mathcal{M}_p: \CC_+\cup \RR^* \to \CC^{N\times N}: z \mapsto \mathcal{M}_p(z),
\eeqs
whose entries are \sn{polynomials on the variables  \ali{$e^{iz r_{j}^+ \ell_j}$, $e^{-izr_{j}^-\ell_j}$}, $j\in I_B$, 
for some positive 
$\ali{r_{j}^+, r_{j}^-}\in \mathbb{Q}$} with coefficients independent of $z$
(hence holomorphic on $\CC$) such that
the   adjugate of $\mathcal{D}(z)$
admits the splitting
\be\label{serge:ass0}
\operatorname{adj}(\mathcal{D}(z))= \mathcal{M}_p(z)+ d_d(z)\mathcal{M}_c(z),  \forall z\in \CC_+\cap \RR^*.
\ee
Further for all $m\in \mathbb{N}$, 
there exists a matrix $\mathcal{M}_m$, such that
\be\label{serge:ass2}
\mathcal{D}(z)^{-1}= \frac{1}{z-\lambda_m}  \mathcal{M}_m+ \mathcal{R}_m(z),
\ee
for all $z$ in a neighborhood of $\lambda_m$, with $\mathcal{R}_m$ holomorphic at $\lambda_m$.
\end{assumption}

\sn{If $S_d$ is empty, this assumption always holds with $d_d=1$
and $\mathcal{M}_p=0$, while \eqref{serge:ass2} is} \ali{unnecessary.}

As  $\mathcal{D}(z)^{-1}= d(z)^{-1} \operatorname{adj}(\mathcal{D}(z))$,   \eqref{serge:ass0} directly implies  that $\mathcal{D}(z)^{-1}$
can be written as
\be\label{serge:ass1}
\mathcal{D}(z)^{-1}= \frac{1}{d(z)}  \mathcal{M}_p(z)+ \frac{1}{d_c(z)}\mathcal{M}_c(z), \forall z\in \CC_+.
\ee
Note further that the assumption \eqref{serge:ass2} implies in particular that
\[
\mathbb{I}=\frac{1}{z-\lambda_m}  \mathcal{M}_m\mathcal{D}(z)+ \mathcal{R}_m(z)\mathcal{D}(z),
\]
for all $z$ in a neighborhood of $\lambda_m$ and consequently
passing to the limit in $z\to \lambda_m$, we find that
\be\label{consequence:serge:ass2}
\mathcal{M}_m\mathcal{D}(\lambda_m)=0.
\ee

Before going on, \sn{we want to show that   $\mathcal{M}_p(z) R F(z,f)$ is  zero
at $\lambda_m$ for data orthogonal to all eigenvectors.}
For that purposes, we first give a relationship between an orthogonality property of
$f$ with the range of the matrix $\mathcal{D}(\lambda_m)$.
\begin{lemma}\label{lortho}
Let $f$ be in  $L^2(\mathcal{R})$ with a compact support.
The system (\ref{systserge}) with $z=\lambda_m$, namely
\be \label{systsergelambdam}
\mathcal{D}(\lambda_m) X_m(f)=b(\lambda_m, f),
\ee
admits a solution if and only if
$f$   is orthogonal to all eigenvectors \sn{$\varphi^{(k)}$  of $H$
associated with the eigenvalue $\lambda_k^2=\lambda_m^2$.}
\end{lemma}
\begin{proof}
As $\mathcal{D}(\lambda_m)$ is a square matrix, the system
\eqref{systsergelambdam} has a solution if and only if 
$b(\lambda_m, f)$ is orthogonal to \ali{$k_{m}$} independent vectors $Y_\ell$, $\ell=1, \cdots \ali{k_{m}}$, where \ali{$k_{m}$} is the dimension of 
$\ker \mathcal{D}(\lambda_m)$, that here coincides with the multiplicity of $\lambda_m^2$ (as eigenvalue of $-\Delta^{Dir}_{\mathcal{R}_B}$ with the additional nullity at the interior nodes belonging to the infinite edges).
As $f$ has a compact support, using \sn{that $b(\lambda_m, f)=R F(\lambda_m, f)$, 
the condition $(b(\lambda_m, f), Y_\ell)_{\mathbb{C}^N}=0$ is equivalent to
\[
(F(\lambda_m, f), R^\top Y_\ell)_{\mathbb{C}^N}=0.
\]
Hence \ali{writing} $R^\top Y_\ell=((c_{\ell, k},d_{\ell,k})_{k\in I_B}, (d_{\ell,k})_{k\in I_U})^\top$,   using the definition of $F(\lambda_m, f),$
and recalling that $\lambda_m$ is positive, this is again equivalent to
\beqs
\sum_{k\in I_B}\left(\overline{c_{\ell, k}} \int_0^{\ell_k} e^{i\lambda_m  y} f_k(y) \,dy
+\overline{d_{\ell, k}} \int_0^{\ell_k} e^{i\lambda_m (\ell_k-y)} f_k(y) \,dy
\right)
\\
+\sum_{k\in I_U} \overline{d_{\ell, k}} \int_0^\infty e^{i\lambda_m  y} f_k(y) \,dy
 =0.
 \eeqs
This is finally equivalent to
\be \label{ortholambdam}
(f, \psi_\ell)_{\mathcal{H}}=0, \forall \ell=1, \cdots, k_m,
\ee
where $\psi_\ell\in L^2_{\rm loc}(\mathcal{R})$ is given by
\beqs
\ali{\psi_{\ell, k}(y)}&=& c_{\ell, k}  e^{-i\lambda_m  y}+d_{\ell, k}  e^{-i\lambda_m \ell_k}  e^{i \lambda_m y} \hbox{ on } (0,\ell_k), \forall k\in I_B,\\
\ali{\psi_{\ell, k}(y)}&=& d_{\ell, k}    e^{-i \lambda_m y} \hbox{ on }(0,\infty), \forall k\in I_U.
\eeqs
This means that $\psi_\ell$ satisfies
the differential equation
\be\label{sn:4/4:1}
\psi_\ell''+\lambda_m^2\psi_\ell=0,
\ee
on all edges of $\mathcal{R}$.
To show that the $\psi_\ell$'s correspond to the eigenvectors of $H$
associated with the eigenvalue $\lambda_m^2$,} it suffices to notice that 
for any $u\in D(H)$ with a compact support, the system
(compare with \eqref{systemresolvante})
\[
\Delta_{{\mathcal R}} u+\lambda_m^2 u=f \hbox{ in } {\mathcal R}
\]
has a solution with $f$ defined as the left-hand side of this identity. Using the procedure of the beginning of this section, the system 
\eqref{systsergelambdam} has a solution and therefore $\psi_\ell$ satisfies
\[
(\Delta_{{\mathcal R}} u+\lambda_m^2 u, \psi_\ell)_{\mathcal{R}}=0,
\]
for any $u\in D(H)$ with a compact support.
By \ali{chosing various} functions $u$, we deduce that
$\psi_\ell$ \sn{satisfies
the differential equation \eqref{sn:4/4:1}}
on all edges of $\mathcal{R}$ as well as the continuity and Kirchhoff law at the interior nodes and Dirichlet boundary conditions at the exterior nodes.
Hence \sn{the proof of Lemma \ref{lspectrediscret} implies that $\psi_\ell=0$ on $\mathcal{R}_U$ and that   indeed the  $\psi_\ell$'s correspond to the eigenvectors of $H$
associated with the eigenvalue $\lambda_m^2$.}
\end{proof}

\begin{corollary}\label{coro:bdX(z,f)}
Let $f\in L^2(\mathcal{R})$ with a compact support be orthogonal to  all eigenvectors \sn{$\varphi^{(k)}$ 
associated with the eigenvalue $\lambda_k^2=\lambda_m^2$.}
Then  
the solution 
$$X(z,f)=((c_k(z,f),d_k(z,f))_{k\in I_B}, (d_k(z,f))_{k\in I_U})^\top$$ of 
\eqref{systserge} \sn{is continuous in   a neighborhood of $\lambda_m$, 
in particular it remains bounded  in   a neighborhood of $\lambda_m$:}
\be\label{bdX(z,f)}
\|X(z,f)\|_2\leq C \|f\|_{L^1(\mathcal{R})},
\ee
with a positive constant $C$ independent of $f$ but depending on $\lambda_m$.
\end{corollary}
\begin{proof}
\ali{Supposing} that  $X_m(f)$ is a  solution of 
\eqref{systsergelambdam}, we notice that
\[
\mathcal{D}(z)^{-1}b(\lambda_m, f)=\mathcal{D}(z)^{-1}\mathcal{D}(\lambda_m) X_m(f).
\]
Then using the assumption \eqref{serge:ass2} as well as \eqref{consequence:serge:ass2} we get
\[
\mathcal{D}(z)^{-1}b(\lambda_m, f)=\mathcal{R}_m(z)\mathcal{D}(\lambda_m) X_m(f),
\]
which means that $\mathcal{D}(z)^{-1}b(\lambda_m, f)$ is holomorphic at $\lambda_m$.

Coming back to $X(z,f)$, we may write
\[
X(z,f)=\mathcal{D}(z)^{-1}b(z, f)=\mathcal{D}(z)^{-1}b(\lambda_m, f)+
\mathcal{D}(z)^{-1}(b(z, f)-b(\lambda_m, f)),
\]
and again using  the assumption \eqref{serge:ass2}, we arrive at
\beqs
\ali{X(z,f)}&=&\mathcal{D}(z)^{-1}b(\lambda_m, f)+\mathcal{R}_m(z)(b(z, f)-b(\lambda_m, f))
\\
&+&
\frac{1}{z-\lambda_m}  \mathcal{M}_m(b(z, f)-b(\lambda_m, f)).
\eeqs
Since $b(z, f)$ is holomorphic at $\lambda_m$, the ratio
\[
\frac{b(z, f)-b(\lambda_m, f)}{z-\lambda_m}
\]
admits a limit as $z$ goes to $\lambda_m$, hence $X(z,f)$ as well.
\end{proof}

\begin{corollary}\label{lemma:annulationMp}
\sn{Let $f\in L^2(\mathcal{R})$ with a compact support be orthogonal to  all eigenvectors $\varphi^{(k)}$ 
associated with the eigenvalue $\lambda_k^2=\lambda_m^2$.
Then 
\[
\mathcal{M}_p(\lambda_m) R F(\lambda_m,f)= 0.
\]}
\end{corollary}
\begin{proof}
 \sn{By
\eqref{serge:ass1}, we have
\[
X(z,f)=\frac{1}{d(z)}  \mathcal{M}_p(z)R F(z,f)+ \frac{1}{d_c(z)}\mathcal{M}_c(z)R F(z,f), \forall z\in \CC_+.
\]
Therefore by  the estimate \eqref{bdX(z,f)} of the previous Corollary and the fact that $d_c(z)$ is different from zero at $\lambda_m$, we have
\[
\left\|\frac{1}{d(z)}  \mathcal{M}_p(z)R F(z,f)\right\|_2\leq C' \|f\|_{L^1(\mathcal{R})},
\]
for all $z$ 
in   a neighborhood of $\lambda_m$,  
with a positive constant $C'$ independent of $f$ but depending on $\lambda_m$. We conclude by passing to the limit as $z\to \lambda_m$,
and reminding that $d(z)\to d(\lambda_m)=0$ as $z\to \lambda_m$.}
\end{proof}

\sn{The previous Corollary shows that for a function $f$
with a compact support orthogonal to all eigenvectors $\varphi^{(m)}$, 
$X(z,f)$ is continuous up the half-line $(0,\infty)$
and consequently 
$K(x, x', z^2) f$ is continuous up to $(0,\infty)$ for such a $f$.
Nevertheless this is not sufficient to  apply Stone's formula and   obtain the resolution of the identity of $H$. 
Hence we will need further assumptions.
For that purposes, let us  denote by  \[
Y_D=\hbox{ \rm Span} \{\varphi^{(m)}\}_{m\in \mathbb{N}},\]
the closed subspace of $L^2(\mathcal{R})$ spanned by the
eigenvectors $\varphi^{(m)}$. 
Further for all $j\in I_U$, we introduce  the closed subspace $Y_j$ of functions in $L^2(\mathcal{R})$ that are zero on $e_k$,
for all $k\in I_B$ and all $k\in I_U\setminus\{j\}$ 
and let $Y_B$ be the orthogonal complement of $\oplus_{i\in I_U} Y_i  \oplus_\perp Y_D$ in $L^2(\mathcal{R})$, so that
\[
L^2(\mathcal{R})=(\oplus_{i\in I_U} Y_i) \oplus_\perp Y_D\oplus_\perp Y_B.
\]
We also denote by $\Pi_j$ (resp. $\Pi_D$, $\Pi_B$) the orthogonal projection on $Y_j$ (resp. $Y_D$, $Y_B$).}

\sn{We are ready to formulate our additional assumptions.
\begin{assumption}\label{lassumptionsergeYj}
For all $k\in I_U$, there exists a mapping
\[
X_k: \CC_+\cup (0,\infty) \to \CC^{N}: z \mapsto X_k(z),
\]
whose entries are polynomials on the variables  $e^{iz r_j^+\ell_j}$, $e^{-iz r_j^-\ell_j}$, $j\in I_B$, \ali{and $e^{\pm iz}$} for some positive $r_j^+, r_j^-\in \mathbb{Q}$, with coefficients independent of $z$ such that
\be\label{hypoYj}
\frac{1}{d(z)}\mathcal{M}_p(z) R F(z,f)=
\frac{1}{2iz d_c(z)} 
\left(\int_0^{\infty} e^{iz y} f_k(y) \,dy\right)
 X_k(z),
\forall f\in Y_k.
\ee
\end{assumption}
Note that this assumption implies that
\be\label{hypoYjcons}
\mathcal{D}(z)^{-1} R F(z,f)=\frac{1}{d_c(z)}
\mathcal{M}_c(z)+\frac{1}{2iz d_c(z)} 
\left(\int_0^{\infty} e^{iz y} f_j(y) \,dy\right)
 X_k(z),
\forall f\in Y_k,
\ee
and consequently on $Y_k$, the kernel $K$ has no more poles at $\lambda_m$ for all $m\in \mathbb{N}$. More precisely, the arguments of the proof of Theorem \ref{theo1} yields
\be\label{resolvantformulaYi} [R(z^2, H)f](x)= \int_{{\mathcal R}} K_k(x, x', z^2) f(x') \; dx', \forall f\in Y_k,
\ee
where the kernel $K_k$ is of the form
\begin{equation}
 K_k(x,y,z^2) = \frac{1}{2iz} \left(e^{iz|x-y|}+ \frac{c^c_{k}(z)}{d_c(z)}     e^{iz (x+y)}\right),  \forall x, y \in  e_k, 
\label{kernel22Yj}
 \end{equation}
\begin{equation}
K_k(x,y,z^2) = \frac{1}{2iz d_c(z)}   c^c_{k',k}(z) e^{iz (x+y)} ,  \forall x\in  e_{k'}, y\in e_{k},   {k'} \in I_U, k\ne k',
\label{kernel21Yj}      
\end{equation}
\begin{equation}
K_k(x,y,z^2) = \frac{1}{2iz d_c(z)} \sum_{\sigma=\pm}   c^c_{k',k,\sigma}(z) e^{iz (\sigma x+y)} ,  \forall x\in e_{k'}, y\in  e_{k},  k'\in I_B,  
\label{kernel12Yj}      
\end{equation}
where the functions $c^c_{k}, c^c_{k',k}$ and $c^c_{k',k,\sigma}$ are \sn{holomorphic functions that are
 polynomials on the variables  $e^{iz r_j^+\ell_j}$, $e^{-iz r_j^-\ell_j}$, $j\in I_B$, for some positive $r_j^+, r_j^-\in \mathbb{Q}$} with coefficients independent of $z$ but that may depend on $k, k', \sigma$ and $\tau$.
}
\sn{
\begin{assumption}\label{lassumptionsergeYB}
It holds
\be\label{hypoYB}
\mathcal{M}_p(z) R F(z,f)=0, 
\forall f\in Y_B.
\ee
\end{assumption}
This assumption directly implies that
\be\label{hypoYBcons}
\mathcal{D}(z)^{-1} R F(z,f)=\frac{1}{d_c(z)}
\mathcal{M}_c(z),
\forall f\in Y_B,
\ee
and consequently on $Y_B$, the kernel $K$ has no more poles at $\lambda_m$ for \ali{any} $m\in \mathbb{N}$, namely
\be\label{resolvantformulaYB} [R(z^2, H)f](x)= \int_{{\mathcal R}} K_B(x, x', z^2) f(x') \; dx', \forall f\in Y_B,
\ee
where the kernel $K_B$ is of the form
\begin{equation}
\label{kernel11B}  K_B(x,y,z^2) =
    \frac{1}{2iz}  \left(e^{iz|x-y|}+ \frac{1}{d_c(z)} \sum_{\sigma=\pm,\, \tau=\pm}   c^c_{B,k,\sigma, \tau}(z) e^{iz (\sigma x+\tau y)}
\right),   \forall x, y \in  e_k, k\in I_B, 
\end{equation}    
\begin{equation}
\label{kernel11bb}  K_B(x,y,z^2) =
    \frac{1}{2izd(z)}   \sum_{\sigma=\pm,\, \tau=\pm}   c^c_{B,k,k',\sigma, \tau}(z) e^{iz (\sigma x+\tau y)},   \forall x \in  e_k, y\in e_{k'}, k, k'\in I_B, k\ne k',
\end{equation}    
\begin{equation}
K_B(x,y,z^2) = \frac{1}{2iz d(z)} \sum_{\tau=\pm}  c^c_{B,k,k',\tau}(z) e^{iz (x+\tau y)} ,  \forall x\in  e_k, y\in e_{k'},  k \in I_U,   k'\in I_B,
\label{kernel21B}      
\end{equation}
where the functions $c^c_{B,k,\sigma, \tau}, c^c_{B,k,k',\sigma, \tau},$ and $c^c_{B,k,k',\tau}$ are  holomorphic functions that are
 polynomials on the variables  $e^{iz r_j^+\ell_j}$, $e^{-iz r_j^-\ell_j}$, $j\in I_B$, 
 for some positive $r_j^+, r_j^-\in \mathbb{Q}$ with coefficients independent of $z$ but that may depend on $k, k', \sigma$ and $\tau$.
}


\sn{Note that the assumptions \ref{lassumptionsergeYj}
and \ref{lassumptionsergeYB} always hold if $S_d$ is empty, with $K_k=K_B=K$, for all $k\in I_U$.}

\sn{We are now able to give the expression of the resolution of the identity $E$ of $H$ on each $Y_k$ under the assumption \ref{lassumptionsergeYj}
and then on $H$ under the assumptions  \ref{lassumptionsergeYj} and \ref{lassumptionsergeYB}.}

\sn{
\begin{theorem} \label{res.idYj}
Let the assumptions \ref{lassumptionserge} and \ref{lassumptionsergeYj}  be satisfied. 
Fix $k\in I_U$ and take $f\in Y_k, g \in {\mathcal H}$ with a compact support and let $0< a < b < + \infty$. Then for any holomorphic  function $h$   on the complex plane, we have
\beq\label{resolidYj}
\nonumber
( h(H) E(a,b)f,g)_{\mathcal H}&=&  -
\frac{1}{\pi}
  \int\limits_{\mathcal R}    \left(\int_{(a,b) } h(\lambda)
        \int\limits_{e_k}  
         \Im K_k(x,x',\lambda)  f_k(x')\; dx'
           \; d\lambda\right) \bar g(x)dx,
\eeq
where for all $\lambda>0$, $K_k(x,x',\lambda)$ was defined  before.
\end{theorem}}

\sn{
\begin{proof}
First by
Stone's formula, we have (see for instance     Lemma~3.13 of \cite{fam2}
or  Proposition 4.5 of \cite{FAMetall10})
\[
(h(H)E(a,b)f, g)_{{\mathcal H}}=\lim_{\delta\rightarrow 0}\lim_{\varepsilon\rightarrow 0}
\frac{1}{2i\pi}
\left( \int_{a+\delta}^{b-\delta}  h(\lambda) 
\left[ R(\lambda-i\varepsilon, H)-R(\lambda+i\varepsilon, H)\right]  d\lambda   f,g\right)_{{\mathcal H}}.
\]
%
As $R(\lambda-i\varepsilon, H)=\overline{R(\lambda+i\varepsilon, H)}$, we can write \ali{thanks to \eqref{resolvantformulaYi}}
\beqs
(h(H)E(a,b)f, g)_{{\mathcal H}}&=&-\frac{1}{\pi}\lim_{\delta\rightarrow 0}\lim_{\varepsilon\rightarrow 0}
\left( \int_{a+\delta}^{b-\delta}   h(\lambda) \Im R(\lambda+i\varepsilon, H)   d\lambda     f,g\right)_{{\mathcal H}}
\\
&=&-\frac{1}{\pi}\lim_{\delta\rightarrow 0}\lim_{\varepsilon\rightarrow 0}
  \int_{a+\delta}^{b-\delta} h(\lambda)(\Im R(\lambda+i\varepsilon, H)      f,g)_{{\mathcal H}}
  d\lambda
  \\
  &=&-\frac{1}{\pi}\lim_{\delta\rightarrow 0}\lim_{\varepsilon\rightarrow 0}
  \int_{a+\delta}^{b-\delta}h(\lambda)
\ali{\Big[}
  \int_{\mathcal R} 
\ali{\Big(}
  \int_{e_k}\Im K_k(x,x',\lambda+i\varepsilon)      f_k(x') 
  dx'
\ali{\Big)} 
  g(x) 
  dx
\ali{\Big]}  
  d\lambda.
\eeqs
At this stage, we take advantage of \ali{\eqref{kernel22Yj}, \eqref{kernel21Yj} and \eqref{kernel12Yj}} to deduce that
\[
K_k(x,y, \lambda+i\varepsilon)\longrightarrow K_k(x,y, \lambda), \hbox{ as } \varepsilon\longrightarrow0.
\]
 Furthermore  we see
 that
 \[
| K_k(x,y, \lambda+i\varepsilon)|\leq \frac{C}{\sqrt{|\lambda+i\varepsilon|}}\leq \frac{C}{\sqrt{a}},
\forall \lambda\in (a,b), x \in e_k,y \in {\mathcal R},
\]
for some positive constant $C$ independent of $x,y$.
 This allows to pass to the limit in $\varepsilon\longrightarrow0$ by using the convergence dominated theorem
 to obtain that
 \beqs
\lim_{\delta\rightarrow 0}
\lim_{\varepsilon\rightarrow 0}
  \int_{a+\delta}^{b-\delta}h(\lambda)
\ali{\Big[}
  \int_{\mathcal R} 
\ali{\Big(}
  \int_{e_k}\Im K_k(x,x',\lambda+i\varepsilon)      f_k(x') 
  dx'
\ali{\Big)} 
  g(x) 
  dx
\ali{\Big]}  
  d\lambda.
\\
  =\int_{a}^{b}h(\lambda)
\ali{\Big[}
  \int_{\mathcal R} 
\ali{\Big(} 
  \int_{e_k}\Im K_k(x,x',\lambda)     f_k(x') 
  dx' 
\ali{\Big)}
  g(x) 
  dx
\ali{\Big]}
  d\lambda.
  \eeqs
\end{proof}
}

\sn{Similar arguments yield.
\begin{theorem} \label{res.idYB}
Let the assumptions \ref{lassumptionserge} and \ref{lassumptionsergeYB}  be satisfied. 
Take $f\in Y_B$ and  $g \in {\mathcal H}$ with a compact support and let $0< a < b < + \infty$. Then for any holomorphic  function $h$   on the complex plane, we have
\beq\label{resolidYB}
\nonumber
( h(H) E(a,b)f,g)_{\mathcal H}&=&  -
\frac{1}{\pi}
  \int\limits_{\mathcal R}    \left(\int_{(a,b) } h(\lambda)
        \int\limits_{\mathcal R_B}  
         \Im K_B(x,x',\lambda)  f_k(x')\; dx'
           \; d\lambda\right) \bar g(x)dx,
\eeq
where for all $\lambda>0$, $K_B(x,x',\lambda)$ was defined  before.
\end{theorem}}

\sn{The previous results directly lead to the expression of the resolution of the identity $E$ of $H$ under the aditional  assumptions  \ref{lassumptionsergeYj} and \ref{lassumptionsergeYB}.
\begin{theorem} \label{res.id}
Let the assumptions \ref{lassumptionserge}, \ref{lassumptionsergeYj} and \ref{lassumptionsergeYB} be satisfied. 
Take $f, g \in {\mathcal H}$ with a compact support and let $0< a < b < + \infty$. Then for any holomorphic  function $h$   on the complex plane, we have
\beq\label{resolid}
\nonumber
( h(H) E(a,b)f,g)_{\mathcal H}&=&  -\sum_{k\in I_U}
\frac{1}{\pi}
  \int\limits_{\mathcal R}    \Big(\int_{(a,b) } h(\lambda)
        \int\limits_{e_k}  
          \Im K_k(x,x',\lambda) (\Pi_k f)_k(x')\; dx'
           \; d\lambda\Big) \bar g(x)dx
           \\
\nonumber
           &-&   
\frac{1}{\pi}
  \int\limits_{\mathcal R}    \Big(\int_{(a,b) } h(\lambda)
        \int\limits_{\mathcal R_B}  
          \Im K_B(x,x',\lambda) (\Pi_B f)(x')\; dx'
           \; d\lambda\Big) \bar g(x)dx          
 \\
           &+&\sum_{m\in \mathbb{N}: a<\lambda_{m}^2<b}  h(\lambda_{m}^2) (f, \ali{\varphi^{(m)}})_{{\mathcal H}} (\ali{\varphi^{(m)}}, g)_{\mathcal H}.
\eeq
\end{theorem}}

\sn{\begin{proof}
We split up
$f$ as follows:
\[
f=\sum_{k\in I_U} \Pi_kf+\Pi_Bf+\ali{\Pi_D}f.
\]
For the first two terms, we can apply the two previous theorems. For the last term,   as $H \varphi^{(n)}= \lambda_m^2\varphi^{(n)}$ for all $m\in \mathbb{N}$, we directly deduce that
\[
( h(H) E(a,b)\ali{\Pi_D} f,g)_{\mathcal H}=
\sum_{m\in \mathbb{N}: a<\lambda_{m}^2<b}  h(\lambda_{m}^2) (f, \ali{\varphi^{(m)}})_{{\mathcal H}} (\ali{\varphi^{(m)}}, g)_{\mathcal H}.
\]
Hence the result.
\end{proof}
}

As direct consequences of the previous Theorem, we get the next results.
\begin{corollary} \label{absence_sing_spec}If the \sn{assumptions \ref{lassumptionserge}, \ref{lassumptionsergeYj} and \ref{lassumptionsergeYB} hold}, then the operator $H$ has no singular \ali{continuous} spectrum and
\[
P_{ac} f=f-\Pi_{D}f, \forall f\in {\mathcal H},
\]
where we recall that
\[
\Pi_{D} f=\sum_{m\in \mathbb{N}} (f, \varphi^{(m)})_{{\mathcal H}} \varphi^{(m)}.
\]
\end{corollary}

\section{The dispersive estimates} \label{proofm}

\begin{theorem} \label{mainresult}
In addition to  the \sn{assumptions \ref{lassumptionserge}
and \ref{lassumptionsergeYj}}, suppose that
there exists a positive constant $\alpha$ such that
\begin{equation}\label{lowerbddc}
|d_c(\mu)|\geq \alpha, \forall \mu\in \mathbb{R}.
\end{equation}
Then  there a positive constant  $C$ such that
\sn{for all $k\in I_U$}, and $t
\neq 0$, \be \label{dispest} \sn{\left\|e^{-itH} \sn{f}
\right\|_{L^\infty({\mathcal R})} \leq C \,
\left|t\right|^{- 1/2}
\|f\|_{
L^1({\mathcal R})}, 
\forall f\in Y_k\cap L^1({\mathcal R}),}\ee where in a natural way
$L^1 (\mathcal{R}) = \ds \ali{\prod_{k \in I}} L^1(e_k),$ and
$L^\infty (\mathcal{R}) = \ds \ali{\prod_{k \in I}} L^\infty(e_k).$
\end{theorem}
\begin{proof}
For any $0<a<b<\infty$,
by Theorem \sn{\ref{res.idYj}}
 for any $x,y\in {\mathcal R}$, we have found the following
 expression for the kernel of the operator
 $e^{-itH}\mathbb{I}_{(a,b)}$ \sn{restricted to $Y_k$, $k\in I_U$}:
$$
\ali{\Big[}
 \int_0^{+ \infty} e^{-it\lambda} \mathbb{I}_{(a,b)}(\lambda) E_{ac}(d\lambda)\ali{\Pi_k}
\ali{\Big]} 
(x,y) 
 =
-\frac{1}{\pi}\int_a^b e^{-it\lambda} \mathbb{I}_{(a,b)}(\lambda) \Im \sn{K_k}(x,y,\lambda) d\lambda,
$$
and by the change of variables $\lambda=\mu^2$, we get
$$
\ali{\Big[}
 \int_0^{+ \infty} e^{-it\lambda} \mathbb{I}_{(a,b)}(\lambda) E_{ac}(d\lambda)\ali{\Pi_k}
\ali{\Big]} 
(x,y) 
=
-\frac{2}{\pi}\int_{\sqrt{a}}^{\sqrt{b}} e^{-it\mu^2}
\mathbb{I}_{(a,b)}(\mu^2) \Im \sn{K_k}(x,y,\mu^2) \mu \
 d\mu.
$$
Now recalling the definition of $\sn{K_k}$, we have to distinguish between the following  cases:
\begin{enumerate}
\item
If $x, y\in \sn{e_k}$, then
$$
2i\mu K_c(x,y,\mu^2)=e^{i\mu|x-y|} + \frac{c^c_k (\mu)}{d_c(\mu)} \, e^{i\mu (x+y)}, k \in I_U.
$$
 Hence in that case, we have to estimate
 \[
\left|
\int_{\sqrt{a}}^{\sqrt{b}} e^{-it\mu^2}   e^{i\mu|x-y|}   d\mu\right|,
\]
and
\[
\left|
\int_{\sqrt{a}}^{\sqrt{b}} e^{-it\mu^2}
 \frac{c^c_k (\mu)}{d_c(\mu)} \,  e^{i\mu (x+y)} d\mu\right|.
\]
The first term is directly estimated by van der Corput Lemma
\cite[Proposition 2, p. 332]{Zygmung:99}:
 \[
\left| \int_{\sqrt{a}}^{\sqrt{b}} e^{-it\mu^2}   e^{i\mu|x-y|}
d\mu\right|\leq   \frac{4 \sqrt{2}}{\sqrt{t}}, \,
\forall \, t
> 0.
\]

For the second term, as $c^c_k (\mu)$ is \sn{a 
 polynomials in the variables  $e^{iz r_j^+\ell_j}$, $e^{-iz r_j^-\ell_j}$, $j\in I_B$,
 for some positive $r_j^+, r_j^-\in \mathbb{Q}$},
with coefficients independent of $\mu$, we are reduced to estimate
\begin{equation}\label{eqsecondterm}
\left|\int_{\sqrt{a}}^{\sqrt{b}} e^{-it\mu^2}
 \frac{1}{d_c(\mu)} \, e^{i\mu (x+y+\ell)} d\mu\right|
\end{equation}
where $\ell$ can take \sn{a finite number of  values   $\ali{r} \ell_k$, $k\in I_B$, with $\ali{r}\in \mathbb{Q}$.}
Now we exploit the assumption \eqref{lowerbddc} and the fact that 
$d_c$ is a quasi-periodic function to deduce by Corollary 1 of \cite[p. 164]{gelfand} that (see also
\cite[Lemma 3.4]{banica})
\[
\frac{1}{d_c(\mu)}=\sum_{\lambda\in D} d_\lambda e^{i\mu \lambda},
\]
where $D$ is a discrete set of real numbers and $d_\lambda\in \mathbb{C}$ for all $\lambda\in D$ such that
\[
\sum_{\lambda\in D} |d_\lambda|<\infty.
\]
This property guarantees that
\[
\left|\int_{\sqrt{a}}^{\sqrt{b}} e^{-it\mu^2}
 \frac{1}{d_c(\mu)} \, e^{i\mu (x+y+\ell)} d\mu\right|
 \leq \sum_{\lambda\in D} |d_\lambda| \left|\int_{\sqrt{a}}^{\sqrt{b}} e^{-it\mu^2}
  e^{i\mu (x+y+\ell+\lambda)} d\mu\right|.
 \]
Again this right-hand side is  estimated by van der Corput Lemma to conclude that
\[
\left|\int_{\sqrt{a}}^{\sqrt{b}} e^{-it\mu^2}
 \frac{1}{d_c(\mu)} \, e^{i\mu (x+y+\ell)} d\mu\right|
 \leq  \frac{4 \sqrt{2}}{\sqrt{t}}
 \sum_{\lambda\in D} |d_\lambda| .
 \]
\medskip
All together we have proved that 
\be
\label{term11}
\left|
\ali{\Big[}
 \int_0^{+ \infty} e^{-it\lambda} \mathbb{I}_{(a,b)}(\lambda) E_{ac}(d\lambda)\ali{\Pi_k}
\ali{\Big]} 
(x,y) 
\right|
\leq
\frac{C}{\sqrt{t}}, \, \forall \, \sn{x , \, y \in e_k}, t > 0,
\ee
where $C$ is a positive constant independent of $x,y,t, a, b$.
\item
If \sn{$x\in \mathcal{R}_U\setminus\{e_k\}$ and $y\in e_k$, then
$$
2i\mu K_k (x,y,\mu^2)= \ds 
\frac{c^c_{k,k^\prime,\sigma} (\mu)}{d_c(\mu)} \, e^{i \mu (x + y)}, 
$$
and the} same arguments as before imply that
\be\label{term21}
\left|\ali{\Big[}
 \int_0^{+ \infty} e^{-it\lambda} \mathbb{I}_{(a,b)}(\lambda) E_{ac}(d\lambda)\ali{\Pi_k}
\ali{\Big]} 
(x,y) \right|
\leq 
\frac{C}{\sqrt{t}}, \forall x\in \sn{ \mathcal{R}_U\setminus\{e_k\},y\in e_k,} \ t>0,
\ee
where $C>0$ is independent of $x,y,t, a, b$.
\item
\sn{If $x\in e_{k'},$ with  $k'\in I_B,$ and $y\in e_k$, then
$$
2i\mu K_c(x,y,\mu^2)= \ds \sum_{\sigma=\pm}   \frac{c^c_{k',k,\sigma}(\mu)}{d_c(\mu)} \, e^{i \mu (x + \tau y)}, 
$$
and we} deduce as before that

\be\label{term12}
\left|
\ali{\Big[}
 \int_0^{+ \infty} e^{-it\lambda} \mathbb{I}_{(a,b)}(\lambda) E_{ac}(d\lambda)\ali{\Pi_k}
\ali{\Big]} 
(x,y) 
\right|
\leq 
\frac{C}{\sqrt{t}}, \forall \sn{y\in e_k,x\in e_{k'}, k'\in I_B,} \, t>0,
\ee
where $C>0$ is independent of $x,y,t, a, b$.
\end{enumerate}

The estimates (\ref{term11}) to (\ref{term12}) directly imply 
that \eqref{dispest} holds.
\end{proof}

Without the assumption \eqref{lowerbddc}, we can prove the same decay but for data in a fixed frequency band, namely we have the following result:

\begin{theorem} \label{mainresultfrequency band}
Let the  \sn{assumptions \ref{lassumptionserge}
and \ref{lassumptionsergeYj}}
be satisfied.
Then  for all for  $ a<b<\infty$, there is a positive constant  $C$ (that may depends on $b$) such that
 \sn{for all $k\in I_U$ and } $t
\neq 0$, \be \label{dispestFB}\sn{ \left\|e^{-itH}\mathbb{I}_{(a,b)}(H) \sn{f}
\right\|_{L^\infty({\mathcal R})} \leq C \,
\left|t\right|^{- 1/2} \sn{\| f\|_{L^1({\mathcal R})}}, \forall f\in Y_k\cap L^1({\mathcal R})}. \ee
\end{theorem}
\begin{proof}
The proof is the same as before except for terms like
\eqref{eqsecondterm}. Indeed in that case, we simply notice that 
$d_c(\mu)$ is non zero on $[0,b]$ and is smooth (in the variable $\mu$), consequently by a variant of 
van der Corput Lemma
\cite[Corollary, p. 334]{Zygmung:99}, we have
\[
\left|\int_{\sqrt{a}}^{\sqrt{b}} e^{-it\mu^2}
 \frac{1}{d_c(\mu)} \, e^{i\mu (x+y+\ell)} d\mu\right|
 \leq  \frac{4 \sqrt{2}}{\sqrt{t}} \left(\frac{1}{|d_c(\sqrt{b})|}+
 \int_{\sqrt{a}}^{\sqrt{b}}  
 \frac{|d_c'(\mu)|}{|d_c(\mu)|^2}  d\mu\right).
\]
This directly implies that
\[
\left|\int_{\sqrt{a}}^{\sqrt{b}} e^{-it\mu^2}
 \frac{1}{d_c(\mu)} \, e^{i\mu (x+y+\ell)} d\mu\right|
 \leq  \frac{C}{\sqrt{t}},
\]
with a positive constant $C$ that may depend on $b$.
\end{proof}

\sn{The previous arguments apply \ali{to the subspace} $Y_B$, namely we have
\begin{theorem} \label{mainresultYB}
In addition to  the \sn{assumptions \ref{lassumptionserge}
and \ref{lassumptionsergeYB}}, suppose that \eqref{lowerbddc} holds.
Then  there is a positive constant  $C$ such that
for all  $t
\neq 0$, \be \label{dispestYB} \sn{\left\|e^{-itH} \sn{f}
\right\|_{L^\infty({\mathcal R})} \leq C \,
\left|t\right|^{- 1/2}
\|f\|_{
L^1({\mathcal R_B})}\leq C \, |\mathcal R_B|^{\frac{1}{2}}
\left|t\right|^{- 1/2}
\|f\|_{
L^2({\mathcal R_B})},  
\forall f\in Y_B},\ee 
where $|\mathcal R_B|$ is the sum of the lengths of the edges of $\mathcal R_B$.
\end{theorem}
\begin{theorem} \label{mainresultfrequency bandYB}
Let the  assumptions \ref{lassumptionserge}
and \ref{lassumptionsergeYB}
be satisfied.
Then  for all $0\leq a<b<\infty$, there a positive constant  $C$ (that may depend on $b$) such that
for all  $t
\neq 0$, \be \label{dispestFBYB}  \left\|e^{-itH}\mathbb{I}_{(a,b)}(H) \sn{f}
\right\|_{L^\infty({\mathcal R})} \leq C \,
\left|t\right|^{- 1/2}  \| f\|_{L^2({\mathcal R}_B)}, \forall f\in Y_B. \ee
\end{theorem}
Let us finish this section with a $(L^1\cap L^2)
-L^\infty$ decay of $e^{itH} P_{ac}$ in $t^{-\frac{1}{2}}$ that directly follows from Theorems \ref{mainresult} and \ref{mainresultYB}.
\begin{corollary}\label{coromainresult}
In addition to  the assumptions \ref{lassumptionserge}, \ref{lassumptionsergeYj}
and \ref{lassumptionsergeYB}, suppose that \eqref{lowerbddc} holds.
Then  there is a positive constant  $C_1$ such that
for all  $t
\neq 0$, 
\beq
\||e^{-itH} P_{ac} f\|_{L^\infty(\mathcal{R})}\leq C_1 |t|^{-1/2} |||f|||,
\forall f\in L^1(\mathcal{R})\cap L^2(\mathcal{R}),
\eeq
where $|||f|||=\|f\|_{L^1(\mathcal{R})}+\|f\|_{L^2(\mathcal{R}_B)}$.
\end{corollary}
}

%

\section{Probability densities and probability flows} \label{probas}

Let us start with the following definition.

\begin{definition}
If $u(t)=e^{-itH}u_0, t\in \mathbb{R}$, with $u_0\in D(H)$, 
then for all $k\in I$,
we define its probability density
$\rho_k$ and probability flow $j_k$ by
\[
\rho_k=|u_k|^2, j_k=\overline{u_k}  u_{k,x} -u_k \overline{u_{k,x}}.
\]   
\end{definition}

We first recall the continuity equations on each edge of the graph.

\begin{lemma}
Let $u_0\in D(H)$ and $u(t)=e^{-itH}u_0, t\in \mathbb{R}$.
Then we have
\begin{eqnarray}\label{conteq:b}
i\frac{d}{dt} \int_{e_k} \rho_k=i\frac{d}{dt} \|u_k\|_{L^2(e_k)}^2
=j_k(0)-j(l_k), \forall k\in I_B,
\\
i\frac{d}{dt} \int_{e_k} \rho_k=i\frac{d}{dt} \|u_k\|_{L^2(e_k)}^2
=j_k(0), \forall k\in I_U.
\label{contequb}
\end{eqnarray}
\end{lemma}
\begin{proof}
These identities as based on simple integration by parts.
Namely as
$u_k$ satisfies
\[
i u_{k,t}=-u_{k,xx} \hbox{ on } e_k,
\]
we have
\begin{eqnarray*}
i\frac{d}{dt} \|u_k\|_{L^2(e_k)}^2&=&
\int_{e_k} (i u_{k,t} \overline{u_k}
-u_k \overline{i u_{k,t}})
\\
&=&
-\int_{e_k} (u_{k,xx} \overline{u_k}
-u_k \overline{u_{k,xx}}).
\end{eqnarray*}
By an integration by parts, we directly find \eqref{conteq:b} and \eqref{contequb}.
\end{proof}

\begin{theorem}
Let $u_0\in D(H)$ and $u(t)=e^{-itH}u_0, t\in \mathbb{R}$.
Then we have
\begin{eqnarray}\label{eq:flowsortant}
i\frac{d}{dt} \sum_{k\in I_B}\int_{e_k} \rho_k=i\frac{d}{dt} \sum_{k\in I_B}\|u_k\|_{L^2(e_k)}^2
=-\sum_{k\in I_U}j_k(0).
\end{eqnarray}
\end{theorem}
\begin{proof}
As $H$ is a selfajoint operator, $\|e^{-itH}\|_{\mathcal{L}(H)}=1$,
consequently
\[
\sum_{k\in I}\|u_k(t)\|_{L^2(e_k)}^2=\sum_{k\in I}\|u_{0,k}\|_{L^2(e_k)}^2, \forall  t\in \mathbb{R}.
\]
This implies that
\begin{eqnarray}\label{eq:somme nulle}
i\frac{d}{dt} \sum_{k\in I_B}\|u_k\|_{L^2(e_k)}^2
=-i\frac{d}{dt} \sum_{k\in I_U}\|u_k\|_{L^2(e_k)}^2,
\end{eqnarray}
and we conclude using \eqref{contequb}.
\end{proof}

The previous theorem means that the time derivative of the square of the $L^2$-norm (of the solution of the Schr\"odinger equation) on the bounded edges of the graph  equals $i \ds \sum_{k\in I_U}j_k(0).$
In other words, the rate of change of the total probability on the bounded part of the graph is equal to $i$ times the probability flow at the connecting vertices.

Let us now show that the hypothesis on $u_0$ implies that
the probability flow remains bounded uniformly in time.

\begin{corollary}
Let $u_0\in D(H)$ and $u(t)=e^{-itH}u_0, t\in \mathbb{R}$.
Then there exists a positive constant $C$ such that  
\begin{eqnarray}\label{eq:borneflowsortant}
\left|\frac{d}{dt} \sum_{k\in I_B}\int_{e_k} \rho_k(t)\right|=\left|\frac{d}{dt} \sum_{k\in I_B}\|u_k(t)\|_{L^2(e_k)}^2\right|
\leq C \|u_0\|_{D(H)}^2, \forall  t\in \mathbb{R}.
\end{eqnarray}
\end{corollary}
\begin{proof}
As $u_0$ belongs to $D(H)$, we have
\[
Hu(t)=e^{-itH}(Hu_0),
\]
and therefore
\[
\|Hu(t)\|_{L^2(\mathcal{R})}=\|Hu_0\|_{L^2(\mathcal{R})}, \forall  t\in \mathbb{R}.
\]
This means that $u_k$ belongs to $H^2(e_k)$, for all $k\in I$
and that there exists a positive constant $C_1$ such that
for all $ t\in \mathbb{R}$, one has
\begin{eqnarray*}
\sum_{k\in I}\|u_k(t)\|^2_{H^2(e_k)}
&\leq &C_1 \sum_{k\in I}(\|u_k(t)\|^2_{L^2(e_k)}+\|Hu_k(t)\|^2_{L^2(e_k)})
\\
&=& C_1
(\|u_0\|^2_{L^2(\mathcal{R})}+
\|Hu_0\|^2_{L^2(\mathcal{R})})=C_1\|u_0\|^2_{D(H)}.
\end{eqnarray*}
Now using the embedding of $H^2(e_k)$ into 
$C^1(\bar e_k)$, we deduce that
there exists a positive constant $C_2$ such that
\[
|j_k(t,0)|\leq C_2 \|u_0\|_{D(H)}^2, \forall k\in I,  t\in \mathbb{R}.
\]
Using this estimate in \eqref{eq:flowsortant}, we obtain 
\eqref{eq:borneflowsortant}, reminding that the number of unbounded edges is finite.
\end{proof}

We end up this section by showing that if the $L^1-L^\infty$ decay is \ali{at least} $|t|^{-1/2}$ 
\ali{
and the assumptions \ref{lassumptionserge},
\ref{lassumptionsergeYj} and \ref{lassumptionsergeYB} are satisfied and \eqref{lowerbddc} holds,
}
then 
the sum of the probability flow at the endpoint of the infinite edges tends to zero as $|t|^{-1}$
for smoother initial data.

\begin{corollary}

\sn{Let the assumptions \ref{lassumptionserge}
and \ref{lassumptionsergeYj}, and \ref{lassumptionsergeYB} be satisfied and suppose that \eqref{lowerbddc} holds.}
Assume that  $u_0\in D(H)\cap L^1(\mathcal{R})$,  that $u_0=P_{ac} u_0$, as well as
$Hu_0\in  L^1(\mathcal{R})$ and let $u(t)=e^{-itH}u_0, t\in \mathbb{R}$. Then
there exists a positive constant $C$ such that  
\begin{eqnarray}\label{eq:flowsortant-tend vers zero}
\quad \quad \left|\frac{d}{dt} \sum_{k\in I_B}\|u_k(t,\cdot)\|_{L^2(e_k)}^2\right|
\leq C |t|^{-1}\sn{|||u_0||| (|||u_0|||+||| H u_0|||)},\forall t\in \mathbb{R}\setminus\{0\}.
\end{eqnarray}
Consequently, for all  $t_2>t_1>0$, one has
\begin{eqnarray}\label{eq:estdel'energie}
&&\left| \sum_{k\in I_B}\|u_k(t_2,\cdot)\|_{L^2(e_k)}^2
 -\sum_{k\in I_B}\|u_k(t_1,\cdot)\|_{L^2(e_k)}^2\right|
 \\
&&\hspace{3cm}\leq  
 C \, \ln \left(\frac{t_2}{t_1}\right) \sn{|||u_0||| \left(|||u_0|||+||| H u_0|||\right)}.
 \nonumber
\end{eqnarray}
\end{corollary}
\begin{proof}
By our assumption on   $u_0$, \sn{Corollary \ref{coromainresult}
yields}
\beq\label{sn:10/10}
\|u(t,\cdot)\|_{L^\infty(\mathcal{R})}&\leq& C_1 |t|^{-1/2} \sn{|||u_0|||},
\\
 \|Hu(t,\cdot)\|_{L^\infty(\mathcal{R})}&\leq& C_1 |t|^{-1/2} 
\sn{|||Hu_0|||},
\label{sn:10/10b}
\eeq
for some positive constant $C_1$ (independent of $t$).
Now fix  $k\in I_U$, then the restriction of $u_k$ on $(0,1)$ satisfies
\[
\|u_k(t,\cdot)\|_{H^2(0,1)}\leq C_2
(\|u_k(t,\cdot)\|_{L^2(0,1)}+\|Hu_k(t,\cdot)\|_{L^2(0,1)}),
\]
for some positive constant $C_2$ (independent of $t$). Hence
\[
\|u_k(t,\cdot)\|_{H^2(0,1)}\leq C_2
\left(\|u_k(t,\cdot)\|_{L^\infty(0,1)}+\|Hu_k(t,\cdot)\|_{L^\infty(0,1)}\right)
\leq C_1^2 |t|^{-1/2}\sn{(|||u_0|||+|||Hu_0|||)}.
\]
By the   Sobolev embedding theorem
we deduce that 
\[
|u_{k,x}(t,0)|\leq C_3 \|u_k(t,\cdot)\|_{H^2(0,1)}
\leq C_3  C_1^2 |t|^{-1/2} \sn{(|||u_0|||+|||Hu_0|||)},
\]
for some positive constant $C_3$ (independent of $t$).

This estimate and \eqref{sn:10/10}-\eqref{sn:10/10b}  directly lead to
\[
|j_k(t,0)| \leq C_4 |t|^{-1} \sn{|||u_0||| (|||u_0|||+||| H u_0|||)},
\forall t>0, k\in I_U,
\]
for some positive constant $C_4$ (independent of $t$).
Using this estimate in \eqref{eq:flowsortant}, we obtain 
\eqref{eq:flowsortant-tend vers zero}.

For the second assertion, for $t_2>t_1>0$, 
as
\begin{eqnarray*}
\left|\sum_{k\in I_B}\|u_k(t_2,x)\|_{L^2(e_k)}^2-
\sum_{k\in I_B}\|u_k(t_1,x)\|_{L^2(e_k)}^2\right|&=&
\left|\int_{t_1}^{t_2}\frac{d}{dt} \sum_{k\in I_B}\|u_k(t,x)\|_{L^2(e_k)}^2\,dt\right|
\\
&\leq&
\int_{t_1}^{t_2}\left|\frac{d}{dt} \sum_{k\in I_B}\|u_k(t,x)\|_{L^2(e_k)}^2\right|\,dt,
\end{eqnarray*}
 we arrive at \eqref{eq:estdel'energie}  using  \eqref{eq:flowsortant-tend vers zero}
\end{proof}

Note that \eqref{eq:somme nulle} and 
\eqref{eq:flowsortant-tend vers zero} directly imply that
\[
\quad \quad \left|\frac{d}{dt} \sum_{k\in I_U}\|u_k(t,\cdot)\|_{L^2(e_k)}^2\right|
\leq C |t|^{-1} \sn{|||u_0||| (|||u_0|||+||| H u_0|||)}, \forall t\in \mathbb{R}\setminus\{0\}.
\]

Note also that the estimate \eqref{eq:estdel'energie} means that the energy on the bounded part of the graph can oscillate between 0 
and the total energy at time $t=0$ but the variation is  at most of logarithmic type.

\section{Some phenomena on special metric graphs} \label{examples}

According to \sn{Corollary \ref{coromainresult}, the $(L^1\cap L^2)
-L^\infty$} decay of $e^{itH} P_{ac}$ in $t^{-\frac{1}{2}}$ holds as far as 
\ali{
the assumption 
\ref{lassumptionserge}
(i.e. the equalities
\eqref{factdetD(z)}, \eqref{serge:ass0},  \eqref{serge:ass2}), 
the assumptions
\ref{lassumptionsergeYj} and \ref{lassumptionsergeYB}
and the inequality
\eqref{lowerbddc} 
}%
are satisfied.
Such assumptions are satisfied when the graph $\mathcal{R}_B$
is a tree with $V_{\rm ext}=\emptyset$ as checked in \cite{banica}. The simplest graph \ali{which} does not enter in this class is the tadpole that we treated in 
\cite{amamnicbis}.

In this section, we analyze different examples of more complex graphs that do not enter into the above class and for which  all these assumptions   hold (or not).

\subsection{$\mathcal R$  is a tree with only one Dirichlet vertex} 

If $\mathcal R$  is a tree with only one Dirichlet vertex (see figure \ref{fig2}), then
$S_d$ is empty, therefore the assumption \rfb{serge:ass0} holds.

\begin{center} 
\includegraphics[scale=0.60]{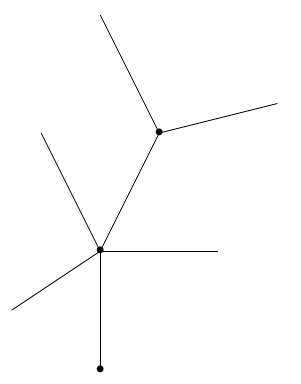}
\captionof{figure}{A tree with only one Dirichlet vertex \label{fig2}}
\end{center}

For this class of examples, the assumption \eqref{lowerbddc}  always holds.

\begin{theorem}
If $\mathcal R$  is a tree with only one Dirichlet vertex, then  the assumption \eqref{lowerbddc} always holds.
\end{theorem}
\begin{proof}
We use the same  iterative argument than in 
 \cite[Proposiition 3.2]{banica}. Namely we show that 
 \beqs
{ \bf P(n)}\quad \hbox{  If }\mathcal R\hbox{ has } n  \hbox{ interior vertices, then it holds }
 \\
 \exists c_{\mathcal R}>0 \hbox{ and }  r_{\mathcal R}\in (0,1):
 |\det \mathcal{D}_{\mathcal R}(z)|\geq c_{\mathcal R},
 \hbox{ and }  \left|\frac{\det \tilde{\mathcal{D}}_{\mathcal R}(z)}{\det  \mathcal{D}_{\mathcal R}(z)}\right|< r_{\mathcal R}, \forall z\in \mathbb{R},
 \eeqs
 where $\mathcal{D}_{\mathcal R}(z)$ is the matrix associated with ${\mathcal R}$ introduced
 in section \ref{sec2}, while  $\tilde {\mathcal{D}}_{\mathcal R}(z)$ is the same matrix \ali{as} $\mathcal{D}_{\mathcal R}(z)$, where in the starting system we replace on one infinite edge the term $d_ke^{izx}$ by
 $\tilde d_ke^{-izx}$. 
 
 Note that the proof of the implication $\bf{P(n)}\Rightarrow \bf{P(n+1)}$ is given in \cite{banica}, hence we only need to check that $\bf{P(1)}$ holds. 
If $ \mathcal R$ has one interior vertex of valency $\ell\geq 3$ (hence one finite edge that we can suppose to be 1 by a scaling argument) and  $\ell-1$  infinite edges, by simple calculations,  we notice that
\[
 \det \mathcal{D}_{\mathcal R}(z)=X^{-1}\{(2-\ell)X^2+\ell\},
 \]
 with $X=e^{iz}$. Therefore as $|X|=1$, we deduce that
 \[
  |\det \mathcal{D}_{\mathcal R}(z)|=|(2-\ell)X^2+\ell|\geq \ell-(\ell-2)=2.
  \]
 Similarly simple calculations yield
 \[
 \det \tilde{\mathcal{D}}_{\mathcal R}(z)=X^{-1}\{(4-\ell)X^2+\ell-3\}.
 \]
 Consequently we have
 \[
 \left|\frac{\det \tilde {\mathcal{D}}_{\mathcal R}(z)}{\det  \mathcal{D}_{\mathcal R}(z)}\right|
 =\frac{|(4-\ell)X^2+\ell-3|}{|(2-\ell)X^2+\ell|}.
 \]
Now using the splitting $e^{iz}=\cos z+i\sin z$
 and some trigonometric formulas, we see that
 \[
\frac{|(4-\ell)X^2+\ell-3|^2}{|(2-\ell)X^2+\ell|^2}=\frac{(\ell-3)^2+2(\ell-3)(4-\ell) \cos z+ (4-\ell)^2}{\ell^2+2\ell(2-\ell) \cos z+ (2-\ell)^2},
\]
 hence we are reduced to estimate
 the ratio
 \[
 r(c)=\frac{(\ell-3)^2+2(\ell-3)(4-\ell) c+ (4-\ell)^2}{\ell^2+2\ell(2-\ell) c+ (2-\ell)^2},
 \forall c\in [-1,1].
 \]
 Since this function $r$ is continuous in $[-1,1]$, it suffices to show that
 \be\label{sn:19/02:1}
 r(c)<1, \forall c\in [-1,1],
 \ee
 to deduce that $\bf{P(1)}$ holds. 
 But this estimate is equivalent to
 \[
 (\ell-3)^2+2(\ell-3)(4-\ell) c+ (4-\ell)^2<\ell^2+2\ell(2-\ell) c+ (2-\ell)^2, \forall c\in [-1,1].
 \]
 And by developping the left-hand side and right-hand side, we get equivalenty
 \[
 c<\frac{10\ell -21}{10\ell -24}.
 \]
 We can conclude that \eqref{sn:19/02:1} holds since $\frac{10\ell -21}{10\ell -24}>1$ for all $\ell\geq 3$.
\end{proof}

\subsection{Bounded subgraph with commensurable lengths\label{sscommensurable}} 

In this case by a scaling argument, we can always assume that the lengths of the edges of the bounded graph $\mathcal{R}_B$   are positive integers and consequently $\det \mathcal{D}(z)$ is a polynomial $P$ \ali{in} the variable $X=e^{iz}$, \sn{up to a factor $e^{-inz}$, for some $n\in \mathbb{N}$.
This also means that $d_c$ is \ali{of} the same form as $\det \mathcal{D}(z)$, while  $d_d$ is   a polynomial $P_d$ \ali{in} the variable $X=e^{iz}$}.
Hence \eqref{factdetD(z)} and \eqref{serge:ass0} always hold, while \eqref{serge:ass2}  holds  if all eigenvalues of $S_d$ are simple. Furthermore   the assumption \eqref{lowerbddc} also holds in such a case, because the zeroes of $P$ are either in the 
unit sphere giving rise to the point spectrum or are outside this unit sphere leading to so-called resonances
but \ali{which} remain bounded away from the unit sphere since only a finite number of such eigenvalues exist. \sn{As a consequence, the zeroes of $P_d$ are  $U_\ell= e^{i\theta_\ell}$
with $\theta_\ell\in [0,2\pi)$, $\ell=1, \cdots, L$ (with $L=0$ if $P$ has no zeroes on the  unit sphere),
and the set $S_d$ is given by
\[
S_d= \ds \cup_{\ell=1}^L \{(\theta_\ell+2k\pi)^2: k\in \mathbb{N}\}.
\]
In such a general configuration, we cannot check the Assumption
\ref{lassumptionsergeYB}, but it can be certainly done for some particular cases, see the examples below. On the other hand, the assumption \ref{lassumptionsergeYj} is always valid if we assume that all eigenvalues of $S_d$ are simple, in other words if $\theta_\ell\ne \theta_{\ell'}$ for $\ell\ne \ell'$.
Indeed by fixing $j\in I_U$,
 we see that
\[
\mathcal{M}_p(z) R F(z,f)=
\frac{1}{2iz}\left(\int_0^{\infty} e^{iz y} f_j(y) \,dy\right)
\mathcal{M}_p(z) R Y_j,
\]
where $Y_j=\left((0, 0))_{k\in I_B}, (\delta_{kj})_{k\in I_U}\right)^\top$. 
But Corollary \ref{lemma:annulationMp} leads to
\[
\left(\int_0^{\infty} e^{i\lambda_m y} f_j(y) \,dy\right)
\mathcal{M}_p(\lambda_m) R Y_j=0, \forall m\in \mathbb{N},
\]
for all $f_j\in L^2_{\rm loc } (0,\infty)$. For each $m$, by choosing $f_j=e^{-i\lambda_m y}\eta(y)$, with $\eta\in \mathcal{D}(0,\infty)$
with $0\leq \eta \leq 1$ and different from zero,  
the first factor is different from zero, and we deduce that
\[
\mathcal{M}_p(\lambda_m) R Y_j=0, \forall m\in \mathbb{N}.
\]
Each component of the vector $\mathcal{M}_p(z) R Y_j$
being a  polynomial \ali{in} the variable $X=e^{iz}$, up to a factor $e^{-inz}$ for some $n\in \mathbb{N}$, we deduce that this polynomial is zero at each $U_\ell$, and therefore divisible by $P_d$ , \ali{which}
 shows that the Assumption \ref{lassumptionsergeYj} holds.
}


Let us now consider some examples where the bounded graph has
a cycle and/or with some incommensurable lengths.

\subsection{A triangle with infinite half-lines at the corners} 

We assume that $\mathcal{R}_B$ is a triangle with three edges of length $1$ and assume that 
an infinite edge is attached at each vertex, see Fig. \ref{fig4}.
In that case, simple calculations show that
\[
S_d = \left\{4\pi^2 n^2, n \in  \sn{\mathbb{N}^*} \right\},
\]
each eigenvalue being simple.  Consequently by subsection \ref{sscommensurable}, all the assumptions \eqref{factdetD(z)}, \eqref{serge:ass0},  \eqref{serge:ass2}, \sn{the assumption \ref{lassumptionsergeYj}} and \eqref{lowerbddc}   hold.

\begin{center} 
\includegraphics[scale=0.60]{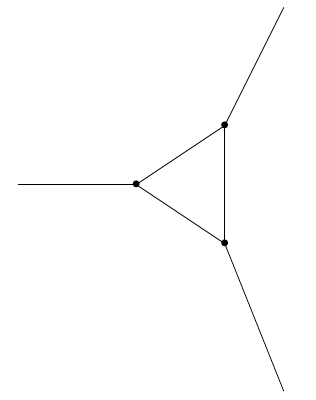}
\captionof{figure}{A triangle graph \label{fig4}}
\end{center}

\subsection{A tadpole with 2 heads \label{ssdouble tadpole}} 

This graph is composed of one head $e_1$ of length $1$ and the second one  $e_2$ of length $\ell$ attached at one point to a half-line, see Figure \ref{fig3}.

\begin{center} 
\includegraphics[scale=0.60]{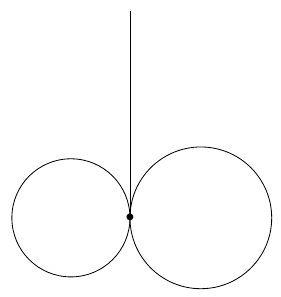}
\captionof{figure}{A tadpole with 2 heads\label{fig3}}
\end{center}

In such a case, it is not difficult to see that
\[
S_d=\{4k^2\pi^2: k\in \mathbb{N}\}\cup \left\{ \frac{4k^2\pi^2}{\ell^2}: k\in \mathbb{N}\right\},
\]
if $\ell$ is an irrational number or if  $\ell=\frac{p}{q}$ is a rational number with $p$ and $q$ two positive relatively prime numbers
such that $p$ or $q$ is even, while 
\[
S_d=\{4k^2\pi^2: k\in \mathbb{N}\}\cup \left\{ \frac{4k^2\pi^2}{\ell^2}: k\in \mathbb{N}\right\}\cup \{((2m-1) q\pi)^2: m\in \mathbb{N}\}
\]
if $\ell=\frac{p}{q}$ is a rational number with $p$ and $q$ two positive relatively prime odd numbers.
For the first sequence,  for all $k\in \mathbb{N}$, the eigenvector $u^{(1,k)}$ associated with $4k^2\pi^2$ is \ali{of} the form  
\[
u^{(1,k)}=(\sin(2k\pi \cdot ),0,0)^\top ,
\]
while for the second sequence, for all $k\in \mathbb{N}$, the eigenvector $u^{(2,k)}$ associated with $\frac{4k^2\pi^2}{\ell^2}$ is of the form  
\[
u^{(2,k)}=\left(0,\sin(\frac{2k\pi}{\ell} \cdot ),0\right)^\top.
\]
In case of a rational number $\ell=\frac{p}{q}$ with $p$ and $q$ two positive relatively prime odd numbers, the eigenvector associated with $((2m-1) q\pi)^2$ is of the form
\[
u^{(3,m)}=( \sin((2m-1) q\pi \cdot), - \sin((2m-1) q\pi \cdot), 0),
\]
hence, this eigenvalue is simple.

Consequently if $\ell$ is a rational number, then the eigenvalues 
$(2k\pi q)^2$ are double and the other ones are simple,
while if $\ell$ is irrational, all eigenvalues are simple.

Let us now check  that 
\ali{
equality \rfb{factdetD(z)} 
of Assumption \ref{lassumptionserge}
}
holds. Indeed 
\sn{by setting $I_B=\{1,2\}$, $I_U=\{3\}$,
$\ell_1=1$ and $\ell_2=\ell$, we see that
\rfb{systserge} holds with}
\[
\mathcal{D}(z)=
\left(
\begin{array}{ccccc}
1&1&0&0&-1
\\
0&0&1&1&-1
\\
1-X&1-X^{-1}&0&0&0
\\
0&0&1-X^\ell&1-X^{-\ell}&0
\\
X-1&1-X^{-1}&X^\ell-1&1-X^{-\ell}&1
\end{array}
\right),
\]
where $X=e^{-iz}$ and for shortness $X^\ell=e^{-i\ell z}$, \sn{
and
\be\label{defbfzdoubletapole}
b(z,f)=
\left(
\begin{array}{ccccc}
F_{3+}(z,f)-F_{1+}(z,f)
\\
F_{3+}(z,f)-F_{2+}(z,f)
\\
F_{1-}(z,f)-F_{1+}(z,f)
\\
F_{2-}(z,f)-F_{2+}(z,f)
\\
F_{1+}(z,f)+F_{1-}(z,f)+F_{2+}(z,f)+F_{2-}(z,f)+F_{3+}(z,f)
\end{array}
\right),
\ee
$F_{j\pm}(z,f)$ being defined \ali{at the beginning of} section 
\ref{sec2}.} Hence we easily check that
\[
\det \mathcal{D}(z)=(1-X) (X^\ell-1) X^{-(1+\ell)} (-3+ X+ X^{\ell}+5 X^{1+\ell}).
\]
If $\ell$ is an irrational number or if  $\ell=\frac{p}{q}$ is a rational number with $p$ and $q$ two positive relatively prime numbers
such that $p$ or $q$ is even,  then \eqref{factdetD(z)} holds with
\[
d_p(z)=(1-e^{-iz}) (e^{-iz\ell}-1),
\]
and 
\[
d_c(z)=e^{iz(1+\ell)} (-3+ e^{-iz}+e^{-iz\ell}+5 e^{-iz(1+\ell)}),
\]
since in this case
the sole roots $z$ of $\det \mathcal{D}(z)$ are the zeroes of 
$e^{-iz}-1$ and of $e^{-iz\ell}-1$. In the case of a rational number
$\ell=\frac{p}{q}$ with $p$ and $q$ two positive relatively prime odd numbers, we set
$Y=e^{-i\frac{z}{q}}$, and notice that
\[
Y^q=e^{-iz} \hbox{ and } Y^p=e^{-iz\ell},
\]
which allows to transform $(-3+ e^{-iz}+e^{-iz\ell}+5 e^{-iz(1+\ell)})$ into 
the polynomial
\[
Q(Y)=-3+ Y^q+ Y^p+5 Y^{p+q}.
\]
Since $Q(-1)=0$, we deduce that $Q$ is divisible by $Y+1$, hence there exists a polynomial $R$ of degree $p+q-1$ such that
 \[
Q(Y)=(Y+1)R(Y).
\]
Furthermore, as $Q'(-1)=-4(p+q)$ is different from zero, $R(-1)$ is different from zero.
Coming back to the original variable $z$, we have found that
\[
-3+ e^{-iz}+e^{-iz\ell}+5 e^{-iz(1+\ell)}
=(1+e^{-i\frac{z}{q}}) R(e^{-i\frac{z}{q}}).
\]
This proves that \eqref{factdetD(z)} holds with
\[
d_d(z)=(1-e^{-iz}) (e^{-iz\ell}-1)(1+e^{-i\frac{z}{q}}),
\]
and 
\[
d_c(z)=e^{iz(1+\ell)} R(e^{-i\frac{z}{q}}).
\]

Direct calculations show that
\[
\operatorname{adj}(\mathcal{D}(z))=X^{-(1+\ell)} \big((X-1) D_1(z)+(1-X^\ell) D_2(z)+(1-X) (X^\ell-1) D_3(z)\big),
\]
with

\[
D_1(z)= 
\left(
\begin{array}{ccccc}
       0   &   0      
          &0                           & 0   &0\\
        0  &  0
        &0	            &       0
      &  0\\
          0      &   0                      &       0  
          &q_{\ell,2}(X)  &   0\\
     0	     &   0                         
     & 0	&s_\ell(X)
      &  0\\
0&0          &    0  
&            0&	0
\end{array}
\right),
\]
\[
D_2(z)= 
\left(
\begin{array}{ccccc}
      0  &    0   
          &q_{\ell,1}(X)                           & 0   &0\\
        0   &   0
        &r_\ell(X)            &       0
      & 0\\
          0	      &   0                      &    0       
          &0	  &   0\\
     0	     &  0                      
     & 0&0   
      &  0\\
0&0           &    0  
&           0&	0
\end{array}
\right),
\]
and
\[
\!\! D_3(z)= \!
\left(
\begin{array}{ccccc}
       3 X^\ell-1   &    2(1- X^\ell)       
          &0                        &  X^\ell-1    & X^\ell+1\\
        X (3X^\ell-1)    &   2X(1-X^\ell)
        &0	            &       X(X^\ell-1)
      &  X(1+ X^\ell)\\
          2(1-X)	      &    3 X-1	                      &        X-1       
          &0  &    X+1\\
     2X^\ell(1-X)	     &   X^\ell (3 X -1)                           
     & X^\ell (X-1)	&0 
      &  X^\ell (X +1)\\
2(X^\ell+1)(1-X) &2(X+1)(1-X^\ell)             &     (X^\ell+1)(X-1)      
&            (X+1)(X^\ell-1)&	(X+1) (X^\ell+1)
\end{array}
\right),
\]
where
$$
q_{\ell,1}(X)=-4   X^{1+\ell}+X^\ell+1, q_{\ell,2}(X)=-X+4   X^{1+\ell}-1,
$$
$$r_\ell(X)= X^2-2 X+  X^{2+\ell}+2   X^{1+\ell},
s_\ell(X)=- X^{1+2\ell}-X^{2\ell}-2  X^{1+\ell}+2 X^\ell.
$$
 
Hence, if $\ell$ is an irrational number or  if  $\ell=\frac{p}{q}$ is a rational number with $p$ and $q$ two positive relatively prime numbers
such that $p$ or $q$ is even,  the assumptions \eqref{serge:ass0} and \eqref{serge:ass2} hold (with 
$\mathcal{M}_c(z)=X^{-(1+\ell)} D_3(z)$ and $\mathcal{M}_p(z)=X^{-(1+\ell)} ((X-1) D_1(z)+(1-X^\ell) D_2(z))$, replacing $X$ (resp. $X^\ell$)
by $e^{-iz}$ (resp. $e^{-i\ell z}$)).

On the contrary if  $\ell=\frac{p}{q}$  is a rational number with $p$ and $q$ two positive relatively prime odd numbers, we need to extract a factor $(1+e^{-i\frac{z}{q}})$ from $D_3(z)$. This is easily obtained again by using the change of unknown $Y=e^{-i\frac{z}{q}}$, and yields
\[
Y^q=e^{-iz}=X \hbox{ and } Y^p=e^{-iz\ell}=X^{\ell}. 
\]
Consequently only the factors $X+1$ and $X^\ell+1$ contribute to this phenomenon since
\beq\label{sn:3/3:1}
X+1=Y^q+1=(Y+1) r_q(Y), 
\\
X^\ell+1=Y^p+1=(Y+1) r_p(Y), 
\label{sn:3/3:2}
\eeq
where
\[
r_p(Y)=\sum_{m=0}^{p-1}(-1)^m Y^{p-1-m}.
\]
This means that 
\[
D_3(z)=D_4(z)+(1+Y) D_5(z),
\]
where
\[
D_4(z)= 
\left(
\begin{array}{ccccc}
       3 X^\ell-1   &    2(1- X^\ell)       
          &0                        &  (X^\ell-1)    & 0\\
        X (3X^\ell-1)    &   2X(1-X^\ell)
        &0	            &       X(X^\ell-1)
      &  0\\
          2(1-X)	      &    (3 X-1)	                      &        (X-1)          
          &0  &   0\\
     2X^\ell(1-X)	     &   X^\ell (3 X -1)                           
     & X^\ell (X-1)	&0 
      &  0\\
0&0            &    0 
&          0&	0
\end{array}
\right),
\]
\[
D_5(z)= 
\left(
\begin{array}{ccccc}
       0  &    0     
          &0                        & 0  &r_p(Y)\\
       0   &   0
        &0	            &      0
      &  Xr_p(Y)\\
          0      &   0                    &        0      
          &0  &    r_q(Y)\\
     0     &   0                        
     & 0	&0 
      &  X^\ell r_q(Y)\\
2r_p(Y))(1-X) &2r_q(Y))(1-X^\ell)            &     r_p(Y)(X-1)      
&            r_q(Y)(X^\ell-1)&	(Y+1)  r_p(Y)r_q(Y)
\end{array}
\right),
\]
And again the assumptions \eqref{serge:ass0} and \eqref{serge:ass2} hold with 
$\mathcal{M}_c(z)=X^{-(1+\ell)} D_5(z)$.

Consequently in the case of a rational number $\ell$ all the assumptions \eqref{factdetD(z)}, \eqref{serge:ass0},  \eqref{serge:ass2}, and \eqref{lowerbddc} are satisfied.

\sn{Let us now check that
Assumption \ref{lassumptionsergeYj} always holds.
Indeed, take $f_1=0$ and $f_2=0$ in 
\rfb{systserge}. In such a case, as $F_{j\pm}(z,f)=0$, 
for $j=1$ and 2,   we see that its \ali{right hand side} 
$b(z,f)$ given by \rfb{defbfzdoubletapole}  reduces to
\[
b(z,f)=F_{3+}(z,f)
\left(
\begin{array}{ccccc}
1
\\
1
\\
0
\\
0
\\
1
\end{array}
\right).
\]
If $\ell$ is an irrational number or  if  $\ell=\frac{p}{q}$ is a rational number with $p$ and $q$ two positive relatively prime numbers
such that $p$ or $q$ is even, we directly deduce that
$\mathcal{M}_p(z)b(z,f)=(0,0,0,0,0)^\top$, which implies that
Assumption \ref{lassumptionsergeYj} \ali{holds} with $X_3(z)=0$.
}

\sn{On the other hand if $\ell=\frac{p}{q}$  is a rational number with $p$ and $q$ two positive relatively prime odd numbers, we see that (with the notation above)
\[
\mathcal{M}_p(z)b(z,f)=\frac{F_{3+}(z,f)}{1+Y} D_4(z)
\left(
\begin{array}{ccccc}
1 \\1 \\0 \\0 \\1
\end{array}
\right)
=
\frac{F_{3+}(z,f)}{1+Y} \left(
\begin{array}{ccccc}
1+X^\ell
\\
X(1+X^\ell)
\\
1+X
\\
X^\ell(1+X)
\\
0
\end{array}
\right).
\]
By the identities \eqref{sn:3/3:1}-\eqref{sn:3/3:2}, we deduce that
\[
\mathcal{M}_p(z)b(z,f)=
F_{3+}(z,f) \left(
\begin{array}{ccccc}
r_p(Y)
\\
Xr_p(Y)
\\
r_q(Y)
\\
X^\ell r_q(Y)
\\
0
\end{array}
\right),
\]
which proves that Assumption \ref{lassumptionsergeYj} holds
with 
\[
X_3(z)=\left(
\begin{array}{ccccc}
r_p(e^{-i\frac{\ell z}{p}})
\\
e^{-iz} r_p(e^{-i\frac{\ell z}{p}})
\\
r_q(e^{-i\frac{\ell z}{p}})
\\
e^{-i\ell z} r_q(e^{-i\frac{\ell z}{p}})
\\
0
\end{array}
\right),
\]
as $Y=e^{-i\frac{z}{q}}=e^{-i\frac{\ell z}{p}}$.}

\sn{Let us now check whether  
assumption 
\ref{lassumptionsergeYB} holds or not.
Indeed, by  taking $f_3=0$ in 
\rfb{systserge} as $F_{3+}(z,f)=0$,    we see that its \ali{right hand side} 
$b(z,f)$  given by \rfb{defbfzdoubletapole}  reduces to
\[
b(z,f)=
\left(
\begin{array}{ccccc}
-F_{1+}(z,f)
\\
-F_{2+}(z,f)
\\
F_{1-}(z,f)-F_{1+}(z,f)
\\
F_{2-}(z,f)-F_{2+}(z,f)
\\
F_{1+}(z),f+F_{1-}(z,f)+F_{2+}(z,f)+F_{2-}(z,f)
\end{array}
\right).
\]}

\sn{If  $\ell$ is an irrational number or  if  $\ell=\frac{p}{q}$ is a rational number with $p$ and $q$ two positive relatively prime numbers
such that $p$ or $q$ is even, we directly deduce that
\[
\mathcal{M}_p(z)b(z,f)=(F_{1-}(z,f)-F_{1+}(z,f))
\left(
\begin{array}{ccccc}
q_{\ell, 1}(X)
\\
r_{\ell}(X)
\\
0
\\
0
\\
0
\end{array}
\right)
+(F_{2-}(z,f)-F_{2+}(z,f))
\left(
\begin{array}{ccccc}
0
\\
0
\\
q_{\ell, 2}(X)
\\
s_{\ell}(X)
\\
0
\end{array}
\right).
\]
}

\sn{Now by the definition of  $F_{j\pm}(z,f)$,  we see that
\[
F_{1-}(z,f)-F_{1+}(z,f)=\frac{1}{2iz} \left( \int_0^1 e^{  i z (1-y)}
f_1(y)dy-\int_0^1 e^{  i z y}
f_1(y)dy\right),
\]
and by the change of variable $x=1-y$ in the first integral, we get
\[
F_{1-}(z,f)-F_{1+}(z,f)=\frac{1}{2iz} \left( \int_0^1 e^{  i z x}
f_1(1-x)dx-\int_0^1 e^{  i z y}
f_1(y)dy\right).
\] 
Similarly, we have
\[
F_{2-}(z,f)-F_{2+}(z,f)=
\frac{1}{2iz} \left( \int_0^\ell e^{  i z x}
f_2(\ell-x)dx-\int_0^\ell e^{  i z y}
f_2(y)dy\right).
\]
Finally in this case $Y_B$ is spanned by the functions
\ali{$$f_{k_1}=(\cos (2k_1\pi x), 0, 0)^\top$$}
and 
\ali{
$$
f_{k_2}= \left(0,\cos (\frac{2k_2\pi x}{\ell}), 0\right)^\top ,
$$
}
with $k_1, k_2\in \mathbb{N}$.
\\
\sn{
Let us prove this assertion:}
\\
\ali{$Y_D$ is spanned by $(\sin(2k_1 .), 0, 0)^\top$ and $(0, \sin(2k_2 ./\ell), 0, 0)^\top), k_1, k_2\in N^*$
and $Y_B$ should be orthogonal to $Y_D$ and to all  fonctions of the form
$(0, 0, f)$ with $f\in L^2(0,\infty)$.
Thus the third component of the elements of $Y_B$ is zero. But the set
$$\{\sin(2k_1 .)\}_{k_1\in N^*}\cup \{\cos(2k_1 .)\}_{k_1\in N}$$
resp.
$$\{\sin(2k_1 ./\ell)\}_{k_1\in N^*}\cup \{\cos(2k_1 ./\ell)\}_{k_1\in N}$$
forms an orthonormal basis of $L^2(0,1)$ (resp. $L^2(0,\ell)$) and therefore we obtain the characterisation of $Y_B$ stated above.
}
\\
%
But for the first family of functions, we see that 
$$F_{2-}(z,f_{k_1})-F_{2+}(z,f_{k_1})=0 , $$
 while
\be\label{sn:8/4:1}
F_{1-}(z,f_{k_1})-F_{1+}(z,f_{k_1})=\frac{1}{2iz} \left( \int_0^1 e^{  i z x}
\cos (2k_1\pi (1-x))dx-\int_0^1 e^{  i z y}
\cos (2k_1\pi y)dy\right)
=0.
\ee
The same phenomenon occurs for the second family, and therefore
Assumption \ref{lassumptionsergeYB} holds.}

\sn{Unfortunately this cancellation phenomenon does not appear 
if $\ell=\frac{p}{q}$  is a rational number with $p$ and $q$ two positive relatively prime odd numbers, hence Assumption \ref{lassumptionsergeYB} does not hold in that case.}

\sn{Now}, if $\ell$ is an irrational number, even if
\sn{$d_c(z)$} is different from zero for all real numbers $z$ (and consequently Assumption \ref{lassumptionserge} holds),
 let us check that \eqref{lowerbddc} does not hold. Indeed by the Theorem from \cite{Hartman-49} (see also \cite[p. 18]{Elsner:96}) with $a=b=1$ and $s=2$,  there   exist two sequences of   positive and odd integers $p_n$ and $q_n$ (that both tend to infinity) such that
\be\label{sn:19/02:2}
\left|\ell-\frac{p_n}{q_n}\right|<\frac{8}{q_n^2}, \forall n\in \mathbb{N}.
\ee
Then  we take
the sequence
\[
z_n=q_n\pi.
\]
With this choice we see that
\[
d_c(z_n)=-4e^{iz_n(1+\ell)} (1+e^{-iz_n\ell}).
\]
But $z_n \ell=q_n\pi \ell$ and multiplying \eqref{sn:19/02:2} by $q_n\pi$, we deduce that
\[
z_n \ell = p_n\pi +r_n,
\]
with $|r_n|<\frac{8\pi}{q_n}$, that hence tends to zero 0 as $n$ goes to infinity.
Consequently, one has 
\begin{equation}
\label{convergence vers zero de d_c}
d_c(z_n)=-4e^{iz_n(1+\ell)} (1-e^{-ir_n}) \to 0, \hbox{ as } n\to \infty.
\end{equation}
 
Before passing to another example, let us study an interesting particular case, the case when the initial data are zero on the two heads. We first state a result for the resolvent.

\begin{theorem}\label{ThmFelix}
Assume that $e_3$ is the infinite edge of the graph.
Assume that $f=(f_1,f_2, f_3)\in L^2(\mathcal{R})$ with $f_1=f_2=0$,
then for any $z\in \mathbb{C}_+$, $u=(z^2-H)^{-1} f$ satisfies
\begin{equation}\label{formuleu3}
u_3(x)=\int_0^\infty K_{33}(x,y, z^2) f_3(y)\,dy,
\end{equation}
where the kernel $K_{33}$ is defined by
\[
K_{33}(x,y, z^2) =\frac{1}{2iz}(e^{-iz|x-y|} +
\gamma(z)    e^{iz(y-x)}),
\]
with  (recalling that $X= e^{-iz}$)
\[
\gamma(z)=\frac{-3X^{1+\ell}+X^{\ell}+X+5}{-3+ X+ X^{\ell}+5 X^{1+\ell}},
\]
 that satisfies
 \begin{equation}\label{valeurabsoluegamma}
|\gamma(z)|=1.
\end{equation}
\end{theorem}
\begin{proof}
The proof of the identity \eqref{formuleu3}
consists in using the approach of section \ref{sec2}
in particular by solving the $5\times 5$ system \eqref{systserge}
with the specify right-hand side.

To show the property \eqref{valeurabsoluegamma} we may notice that
\[
\gamma(z)=X^{1+\ell}
\frac{-3+X^{-1}+X^{-\ell}+5 X^{1+\ell}}{-3+ X+ X^{\ell}+5 X^{1+\ell}}.
\]
Since $X^{-1}= e^{iz}$, we remark that $X^{-1}= \bar X$, therefore
\[
-3+X^{-1}+X^{-\ell}+5 X^{1+\ell}=
\overline{-3+ X+ X^{\ell}+5 X^{1+\ell}},
\]
in other word the numerator of the previous fraction is the conjuguate of its denominator. This directly leads to 
 \[
|\gamma(z)|=|X^{1+\ell}|=1.
\]
\end{proof}

The property \eqref{valeurabsoluegamma} allows us to conjecture that the $L^1-L^\infty$ decay in $|t|^{-1/2}$ for the third component of the solution of Schr\"odinger equation with  an initial datum that is zero on the heads remains valid. Since we need to estimate the quantity
(see the arguments of section 3)
\[
\left|\int_{\sqrt{a}}^{\sqrt{b}} e^{it \mu^2} \gamma(\mu))    e^{i\mu(y-x)}\,d\mu\right|,
\]
the  variant of 
van der Corput Lemma
\cite[Corollary, p. 334]{Zygmung:99} is useless in this case since 
the derivative of $\gamma$ is not integrable due to the property \eqref{convergence vers zero de d_c}, therefore another tool should be used to prove this conjecture.

Note that the previous theorem also allows to give an  explicit expression of the solution of the Schr\"odinger equation on the graph on the infinite edge once the intial data on the heads are   zero and the one  on the infinite edge is  compactly supported.
Namely we have

\begin{theorem}\label{Thm5Felix}
Assume that $e_3$ is the infinite edge of the graph.
Assume that $u_0=(0,0, u_{0,3})\in L^2(\mathcal{R})$ with $u_{0,3}
\in \mathcal{D}(0,\infty)$,
then the solution $u=(u_1,u_2,u_3)$ of Schr\"odinger equation on the graph  $G$ satisfies
\begin{equation}\label{formuleu3Schrodinger}
u_3(t, x)=\mathcal{F}^{-1}\left(e^{-i t \cdot^2}(1+\gamma)
\mathcal{F} (\tilde u_{0,3})\right)(x), \forall x>0, t\in \mathbb{R},
\end{equation}
where $\tilde u_{0,3}$ means the extension of $u_{0,3}$ by zero in $(-\infty,0)$ and
\[
\mathcal{F} (\tilde u_{0,3})(z)=\frac{1}{\sqrt{2\pi}}
\int_\mathbb{R} e^{-izx} \tilde u_{0,3}(x)\,dx, \forall z\in \mathbb{R}.
\]
denotes the usual  Fourier transform of $\tilde u_{0,3}$.
\end{theorem}
\begin{proof}
For $0<a<b$, we have
\[
(\mathbb{I}_{(a,b)} u)(t)=(e^{-itH}\mathbb{I}_{(a,b)})u_0,
\]
and therefore by Theorem \ref{res.id}, we get
\[
(\mathbb{I}_{(a,b)} u)_3(t,x)=
-\frac{2}{\pi}\int_{\sqrt{a}}^{\sqrt{b}} e^{-it\mu^2}
\int_0^\infty 
\Im K_{33}(x,y,\mu^2) u_{0,3}(y)\, dy
\mu \, d\mu.
\]
Hence using the expression of the kernel \eqref{formuleu3}, we get
\beqs
(\mathbb{I}_{(a,b)} u)_3(t,x)&=&
\frac{1}{\pi}\int_{\sqrt{a}}^{\sqrt{b}} e^{-it\mu^2}
\int_0^\infty 
u_{0,3}(y)
\big\{\cos(\mu(x-y)) (1+\Re \gamma(\mu))
\\
&&\hspace{2cm}-
\sin(\mu(x-y)) \Im \gamma(\mu)\big\} dy \,
 d\mu.
\eeqs

By using the expression of the cosinus and the sinus in term of the imaginary powers, we see that
\beqs
\int_0^\infty 
u_{0,3}(y)
 \cos(\mu(x-y))dy&=&
 \frac{\sqrt{2\pi}}{2}\left(e^{i\mu x} \mathcal{F} (\tilde u_{0,3})(\mu)+
 e^{-i\mu x} \mathcal{F} (\tilde u_{0,3})(-\mu)\right),
 \\
 \int_0^\infty 
u_{0,3}(y)
 \cos(\mu(x-y))dy&=&
 \frac{\sqrt{2\pi}}{2i} \left(e^{i\mu x} \mathcal{F} (\tilde u_{0,3})(\mu)-
 e^{-i\mu x} \mathcal{F} (\tilde u_{0,3})(-\mu)\right).
 \eeqs
Inserting these two expressions in the previous one, we obtain
\beqs
(\mathbb{I}_{(a,b)} u)_3(t,x)&=&
\frac{1}{\sqrt{2\pi}}  
\int_{\sqrt{a}}^{\sqrt{b}} e^{-it\mu^2} \Big(
(1+\Re \gamma(\mu))
 \left(e^{i\mu x} \mathcal{F} (\tilde u_{0,3})(\mu)+
 e^{-i\mu x} \mathcal{F} (\tilde u_{0,3})(-\mu)\right)
\\
&&\hspace{2cm}+i \Im \gamma(\mu)
\left(e^{i\mu x} \mathcal{F} (\tilde u_{0,3})(\mu)-
 e^{-i\mu x} \mathcal{F} (\tilde u_{0,3})(-\mu)\right) \Big) \,
 d\mu\\
 &=&
\frac{1}{\sqrt{2\pi}}  \Big(
\int_{\sqrt{a}}^{\sqrt{b}} e^{-it\mu^2}  
(1+\Re \gamma(\mu))
 e^{i\mu x} \mathcal{F} (\tilde u_{0,3})(\mu)\,d\mu
 \\
 &+&\int_{\sqrt{a}}^{\sqrt{b}} e^{-it\mu^2}  
(1+\Re \gamma(\mu))
 e^{-i\mu x} \mathcal{F} (\tilde u_{0,3})(-\mu)\,d\mu
\\
&+&i\int_{\sqrt{a}}^{\sqrt{b}} e^{-it\mu^2}   \Im \gamma(\mu)
e^{i\mu x} \mathcal{F} (\tilde u_{0,3})(\mu)\,
 d\mu
\\
&-&i\int_{\sqrt{a}}^{\sqrt{b}} e^{-it\mu^2}   \Im \gamma(\mu)
 e^{-i\mu x} \mathcal{F} (\tilde u_{0,3})(-\mu)\,
 d\mu \Big) .
\eeqs
Performing the change of variable $z=-\mu$ in the second and fourth integrals, we find
\beqs
(\mathbb{I}_{(a,b)} u)_3(t,x)&=&
\frac{1}{\sqrt{2\pi}}  \Big(
\int_{\sqrt{a}}^{\sqrt{b}} e^{-it\mu^2}  
(1+\Re \gamma(\mu))
 e^{i\mu x} \mathcal{F} (\tilde u_{0,3})(\mu)\,d\mu
 \\
 &+&\int_{-\sqrt{b}}^{-\sqrt{a}} e^{-itz^2}  
(1+\Re \gamma(-z)
 e^{iz x} \mathcal{F} (\tilde u_{0,3})(z)\,dz
\\
&+&i\int_{\sqrt{a}}^{\sqrt{b}} e^{-it\mu^2}   \Im \gamma(\mu)
e^{i\mu x} \mathcal{F} (\tilde u_{0,3})(\mu)\,
 d\mu
\\
&-&i\int_{-\sqrt{b}}^{-\sqrt{a}}e^{-itz^2}   \Im \gamma(-z)
 e^{iz x} \mathcal{F} (\tilde u_{0,3})(z)\,
 dz \Big).
\eeqs
By noticing that $\gamma(-z)=\overline{\gamma(z)}$, we then have
\[
\Re\gamma(-z)=\Re\gamma(z), \quad \Im\gamma(-z)=-\Im\gamma(z),
\] and consequently we find
\beqs
(\mathbb{I}_{(a,b)} u)_3(t,x)&=&
\frac{1}{\sqrt{2\pi}}  \Big(
\int_{(\sqrt{a},\sqrt{b})\cup (-\sqrt{b},-\sqrt{a})} e^{-it\mu^2}  
(1+\Re \gamma(\mu))
 e^{i\mu x} \mathcal{F} (\tilde u_{0,3})(\mu)\,d\mu
\\
&+&i\int_{(\sqrt{a},\sqrt{b})\cup (-\sqrt{b},-\sqrt{a})} e^{-it\mu^2}   \Im \gamma(\mu)
e^{i\mu x} \mathcal{F} (\tilde u_{0,3})(\mu)\,
 d\mu\Big)
 \\
 &=&
 \frac{1}{\sqrt{2\pi}} \int_{(\sqrt{a},\sqrt{b})\cup (-\sqrt{b},-\sqrt{a})} e^{-it\mu^2}  (1+ \gamma(\mu))
 e^{i\mu x} \mathcal{F} (\tilde u_{0,3})(\mu)\,d\mu.
\eeqs
Since $\mathcal{F} (\tilde u_{0,3})$ is rapidly decreasing, hence integrable in $\mathbb{R}$ and due to the property \eqref{valeurabsoluegamma}, we can pass to the limit as $a$ goes to 0 
and $b$ goes to infinity and find \eqref{formuleu3Schrodinger}.
\end{proof}

Note that passing to the limit in \eqref{formuleu3Schrodinger} as $t$ goes to zero, we find
\[
u_{0,3}(x)=u_3(0, x)=\mathcal{F}^{-1}\left((1+\gamma)
\mathcal{F} (\tilde u_{0,3})\right)(x), \forall x>0,
\]
and since 
\[
\mathcal{F}^{-1}\left((1+\gamma)
\mathcal{F} (\tilde u_{0,3})\right)(x)=
u_{0,3}(x)+\mathcal{F}^{-1}\left(\gamma
\mathcal{F} (\tilde u_{0,3})\right)(x),  \forall x>0,
\]
we deduce that
\[
\mathcal{F}^{-1}\left(\gamma
\mathcal{F} (\tilde u_{0,3})\right)(x)=0,  \forall x>0,
\]
for any $u_{0,3}\in \mathcal{D}(0,\infty).$
This is a surprising property that is not easy to check using the form of $\gamma$ but is true due to our previous Theorem.

\ali{
The fact that the coefficient $\gamma(\mu)$, which carries all the information of the influence of the heads of the double tadpole on its queue for the frequency $\mu$, has the remarkable property  
$|\gamma(\mu)| = 1 \ \ \forall \mu \in \mathbb{C}$, should be the consequence of hidden structures worth to be revealed. The analogous phenomenon occurs for the Y-graph.
}
\subsection{A spider graph}

This graph corresponds to the case of $\mathcal{R}_B$ equal to a circle attached with $N\in \mathbb{N}$ half-lines at 0, see  Fig. \ref{figspider}.  Let us notice that the case $N=1$ corresponds to the tadpole \cite{amamnicbis}. By a scaling argument, we can assume that the length of the circle is 1. For the same reason than in the case of the tadpole,  
\[
S_d=\{(2k\pi)^2\}_{k\in \mathbb{N}},
\]
each eigenvalue being simple. Consequently by subsection \ref{sscommensurable}, all the assumptions \eqref{factdetD(z)}, \eqref{serge:ass0},  \eqref{serge:ass2},   and \eqref{lowerbddc}   hold.
Note that \sn{we can check that the assumptions \ref{lassumptionsergeYj} and \ref{lassumptionsergeYB}  hold. In this last case, we use the same arguments \ali{as} in subsection \ref{ssdouble tadpole} because $Y_B$ is spanned by the set $\{\cos(2k\pi x): k\in \mathbb{N}\}$,
see the property \eqref{sn:8/4:1}.
} 

\begin{center} 
\includegraphics[scale=0.60]{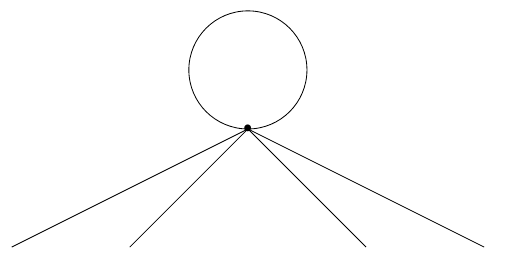}
\captionof{figure}{A spider graph \label{figspider}}
\end{center}

\subsection{A $Y$ graph}
Here we assume that  $\mathcal{R}_B$ is made of two intervals, one of lenght 1 and the other one of lenght
$\ell$, and that an infinite edge is attached at one of their extremities  (the sole vertex of the graph), see  Fig. \ref{fig5}.

\begin{center} 
\includegraphics[scale=0.60]{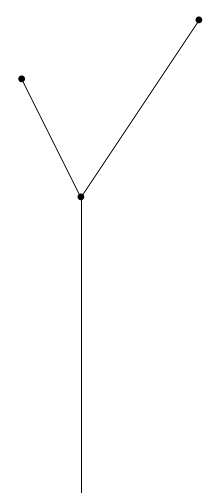}
\captionof{figure}{A $Y$ graph \label{fig5}}
\end{center}

In this case, one can show that
\[
\det \mathcal{D}(z)=e^{i(1+\ell)z} (1+e^{-2iz}+e^{-i2\ell z}-3e^{-i2(1+\ell)z}),
\]   
and
$S_d=\emptyset$ if $\ell$ is an irrational number, otherwise 
\[
S_d=\{(mq\pi)^2\}_{m\in \mathbb{N}},
\]
each eigenvalue being simple,
if $\ell=\frac{p}{q}$ is a rational number with $p$ and $q$ two positive relatively prime numbers.

Hence if $\ell$ is a rational number, by subsection \ref{sscommensurable}
and the simplicity of the eigenvalues, all the assumptions \eqref{factdetD(z)}, \eqref{serge:ass0},  \eqref{serge:ass2}, and \eqref{lowerbddc}   hold.
\sn{Furthermore we can also show that the
Assumption \ref{lassumptionsergeYj} with $X_3(z)=0$, while the assumption \ref{lassumptionsergeYB} does not hold.}

On the contrary, if $\ell$ is an irrational number, even if
$d(z)=\det \mathcal{D}(z)$ is different from zero for all real numbers $z$ (and consequently Assumption \ref{lassumptionserge} holds),
 let us check that \eqref{lowerbddc} does not hold.  Indeed in this case, by Liouville's theorem, there always exist two sequences of   positive integers $p_n$ and $q_n$ (that both tend to infinity as $n$ goes to infinity) such that
\begin{equation}\label{Liouville}
\left|\ell-\frac{p_n}{q_n}\right|<\frac{1}{q_n^2}, \forall n\in \mathbb{N}.
\end{equation}
Let us now define
the sequence
\[
z_n=\frac{p_n\pi}{\ell}.
\]
With this choice we see that
\[
d(z_n)=2(-1)^{p_n}(1-e^{2iz_n}).
\]
But multiplying \eqref{Liouville} by $\frac{q_n\pi}{\ell}$, we get
\[
|q_n\pi-z_n|<\frac{\pi}{\ell q_n},
\]
which means that we can write
$z_n=q_n \pi+r_n$,
with 
\[
|r_n|<\frac{\pi}{\ell q_n},
\]
that tends to zero as $n$ goes to infinity. This implies that
\[
d(z_n)=2(-1)^{p_n}(1-e^{2ir_n}),
\]
that tends to zero as $n$  goes to infinity. 

Note that a similar phenomenon that the one described in Theorem \ref{ThmFelix} appears if  $\ell$ is an irrational number.

\section {Conclusion}

For   graphs for which some of our assumptions \eqref{factdetD(z)}, \eqref{serge:ass0},  \eqref{serge:ass2}, and \eqref{lowerbddc} do not hold, the question of the   $L^1-L^\infty$ decay of $e^{itH} P_{ac}$ remains open. We believe that \eqref{factdetD(z)}, \eqref{serge:ass0}, and \eqref{serge:ass2} are always satisfied, while we have seen that this is not always the case for \eqref{lowerbddc}. We actually conjecture that when \eqref{lowerbddc} does not hold, the decay in $t^{-\frac{1}{2}}$ is no more valid but deteriorates into $t^{-s}$, with $0<s<\frac{1}{2}$.

\end{document}